\documentclass[10pt,a4paper]{article}
\usepackage{latexsym,amssymb,amsmath,amsthm,amsfonts}
\usepackage{graphicx}
\usepackage[all]{xy}

\makeatletter

\newcommand{\Rmnum}[1]{\expandafter\@slowromancap\romannumeral #1@}
\makeatother

\newcommand{\ud}{\mathrm{d}}

\newcommand{\f}{\frac}

\newcommand{\pppp}[4]%
  {\frac{\partial^3{#1}}{\partial{#2}\partial{#3}\partial{#4}}}

\newcommand{\yI}{y^i}
\newcommand{\yk}{y_k}

\renewcommand{\a}{\alpha}
\renewcommand{\b}{\beta}
\newcommand{\ab}{(\alpha,\beta)}

\newcommand{\ba}{\bar\alpha}
\newcommand{\bb}{\bar\beta}

\newcommand{\aij}{a_{ij}}

\newcommand{\bij}{b_{i|j}}

\newcommand{\baij}{\bar a_{ij}}

\newcommand{\bi}{b_i}
\newcommand{\bj}{b_j}
\newcommand{\bk}{b_k}
\newcommand{\bI}{b^i}
\newcommand{\bK}{b^k}

\newcommand{\bbi}{\bar b_i}

\newcommand{\rij}{r_{ij}}

\newcommand{\brij}{\bar r_{ij}}
\newcommand{\rIi}{r^i{}_i}

\newcommand{\rIk}{r^i{}_k}
\newcommand{\roo}{r_{00}}
\newcommand{\rko}{r_{k0}}
\newcommand{\rIo}{r^{i}{}_{0}}
\newcommand{\ri}{r_i}
\newcommand{\rj}{r_j}
\newcommand{\rI}{r^i}
\newcommand{\rk}{r_k}

\newcommand{\ro}{r_0}

\newcommand{\roohk}{r_{00|k}}
\newcommand{\roohb}{r_{00|b}}
\newcommand{\rkoho}{r_{k0|0}}

\newcommand{\rIhk}{r^i{}_{|k}}
\newcommand{\rIho}{r^i{}_{|0}}
\newcommand{\rIhi}{r^i{}_{|i}}

\newcommand{\roho}{r_{0|0}}
\newcommand{\rohk}{r_{0|k}}
\newcommand{\rkho}{r_{k|0}}

\newcommand{\rohb}{r_{0|b}}

\newcommand{\rhk}{r_{|k}}
\newcommand{\rhb}{r_{|b}}

\newcommand{\sij}{s_{ij}}

\newcommand{\sIo}{s^i{}_0}
\newcommand{\sko}{s_{k0}}

\newcommand{\bsij}{\bar s_{ij}}
\newcommand{\si}{s_i}
\newcommand{\sj}{s_j}
\newcommand{\sk}{s_k}
\newcommand{\sI}{s^i}

\newcommand{\bsi}{\bar s_i}

\newcommand{\so}{s_0}
\newcommand{\bso}{\bar  s_0}

\newcommand{\sIohi}{s^i{}_{0|i}}

\newcommand{\sIohk}{s^i{}_{0|k}}

\newcommand{\sIkho}{s^i{}_{k|0}}

\newcommand{\sIoho}{s^i{}_{0|0}}
\newcommand{\sIhi}{s^i{}_{|i}}

\newcommand{\sIhk}{s^i{}_{|k}}

\newcommand{\sIho}{s^i{}_{|0}}
\newcommand{\soho}{s_{0|0}}
\newcommand{\sohb}{s_{0|b}}

\newcommand{\sohk}{s_{0|k}}

\newcommand{\skho}{s_{k|0}}

\newcommand{\poo}{p_{00}}

\newcommand{\pk}{p_k}
\newcommand{\po}{p_0}

\newcommand{\qk}{q_k}

\newcommand{\qI}{q^i}

\newcommand{\qoo}{q_{00}}

\newcommand{\qko}{q_{k0}}
\newcommand{\qok}{q_{0k}}
\newcommand{\qo}{q_0}
\newcommand{\qqk}{q^*{}_k}

\newcommand{\qqo}{q^*{}_0}

\newcommand{\too}{t_{00}}
\newcommand{\btoo}{\bar t_{00}}

\newcommand{\tIo}{t^i{}_{0}}

\newcommand{\tk}{t_k}

\newcommand{\tI}{t^i}
\renewcommand{\to}{t_0}
\newcommand{\tIk}{t^i{}_{k}}
\newcommand{\tIi}{t^i{}_{i}}

\newcommand{\dIk}{\delta^i{}_k}

\newcommand{\RIk}{{R}^i{}_k}

\newcommand{\bRIk}{{}^{\bar\alpha} \bar{R}^i{}_k}

\newcommand{\Rbb}{R_{bb}}
\newcommand{\bRbb}{\bar R_{\bar b\bar b}}

\newcommand{\Ricoo}{{Ric_{00}}}

\newcommand{\bRicoo}{{}^{\bar{\alpha}}{{Ric_{00}}}}

\newcommand{\Rico}{{Ric_{0}}}

\newcommand{\Ric}{{Ric}}

\allowdisplaybreaks[4]
\newtheorem{lemma}{Lemma}[section]
\newtheorem{proposition}[lemma]{Proposition}
\newtheorem{theorem}[lemma]{Theorem}

\newtheorem{example}[lemma]{Example}
\newenvironment{remark}{\textbf{Remark}}{}

\numberwithin{equation}{section}

\begin{document}
\title{Deformations and Einstein metrics \Rmnum{1}}

\author{Changtao Yu}
\date{}
\maketitle

\begin{abstract}
	This essay is about how to construct a new Einstein metric by an old one. Given an Einstein metric $\a$ and its Killing $1$-form $\b$, donote $b:=\|\b\|_{\a}$, we aim to determined the deformation factors $e^{\rho(b^2)}$ and $\kappa(b^2)$ such that $e^{\rho(b^2)}\sqrt{\a^2-\kappa(b^2)\b^2}$ becomes an Einstein metric. In face, it will depends critically on the peculiarities of the Killing $1$-form. As the first article in this series, we assume $\b$ satisfies two curcial conditions (\ref{owemigneigaindna}) and (\ref{tuabgengand}), which are simple, natural and occursing only on even-dimensional manifolds. In this essay, we just need to regard the metric as a quadratic form. Any other additional structure on manifolds, such as topological structure, complex structure, etc., are not used.
\end{abstract}

\section{Introduction}
The theme of this essay is not about Finsler geometry, but talking something about this subject may be useful to understand what are we going to do.  In fact, the crucial method in our discussion, namely a special kind of metrical deformations, arises from the resent research on Finsler geometry.

Some Finsler metrics are determined by a Riemannian metric $\a$ and a $1$-form $\b$. Such kind of Finsler metrics is much closeing to Riemannian metrics\cite{yct-zhm-pfga}. For example, take $\a=|y|$, $\b=\langle x,y\rangle$ and denote $b:=\|\b\|_{\a}$, then
$F=\f{\sqrt{(1-b^2)\a^2+\b^2}}{1-b^2}+\f{\b}{1-b^2}$ is the famous Funk metric on $\mathbb{B}^n$\cite{funk-uber,db-robl-szm-zerm}. Another two well-known metrics, which are found by L. Berwald\cite{berw-uber} and R. Bryant\cite{brya-some} respectively, are $F=\f{(\sqrt{(1-b^2)\a^2+\b^2}+\b)^2}{(1-b^2)^2\sqrt{(1-b^2)\a^2+\b^2}}$ defined on $\mathbb{B}^n$ and $F=\Re\f{\sqrt{(e^{ip}+b^2)\a^2+\b^2}-i\b}{e^{ip}+b^2}$ defined on $\mathbb{S}^n$. All of them are of excellent geometrical properties: complete or at least forward complete, locally projectively flat, and of constant flag curvature\cite{yct-zhm-pfga}.

From an algebraic point of view, a Riemannian metric is just a quadratic form. Hence, if you have a Riemannian metric $\a$ and a $1$-form $\b$, you can obtain infinity many Riemannian metrics by using $\b$ to deform $\a$ as below
\begin{eqnarray*}
	\ba=e^{\rho}\sqrt{\a^2-\kappa \b^2}.
\end{eqnarray*}
The factors $\kappa$ and $\rho$ could depend on the points of the given manifold. However, based on our earlier researches, we will just consider a simple case, namely they depend only on the length of $\b$. Such special kind of metrical deformations, introduced by the author, are called $\b$-deformations\cite{yct-dhfp}.

$\b$-deformations are used successfully to solve some problems in Finsler geometry. However, such technique relates only to Riemannian metrics. Hence, we aim to explore the effect of $\b$-deformations in Riemannian geometry, and this essay is just the beginning.

Many facts declare a simple phenomenon, that is, good data produce good structures. Our question is very simple. Given an Einstein metric $\a$ and a special $1$-form $\b$, can we obtain some new Einstein metrics by using data $\ab$ combining with $\b$-deformations?

The answer is positive, as we expected at the very beginning of this project. In fact, we will choose $\a$ as a Riemannian metric with constant sectional curvature or even an Euclidean metric, and $\b$ as simple as possible. Even so, we can obtain infinity many Einstein metrics by this new way, some of them are known, and some of them are new.

Otherwise, perhaps any Einstein metric can be rebuilt by a similar way. We need to explain it a little bit more.

Consider a Riemannian metric $\a$ and some $1$-forms $\b_1$, $\b_2$, $\cdots$, $\beta_m$. Let $b_{kl}=\langle\beta_k,\beta_l\rangle_\a$, $\rho$ and $\kappa_{ij}$ be functions of all the $b_{kl}$. Then
\begin{eqnarray*}
	\ba=e^{\rho}\sqrt{\a^2-\sum\kappa_{ij}\b_i\b_j}
\end{eqnarray*}
is a Riemannian metric when the factors $\kappa_{ij}$ satisfy some necessary restrictions to make sure the positive definiteness of $\ba$. Such kind of deformations can be regarded as {\it multiple $\b$-deformations}.

By this way, one can obtain many Riemannian metric of excellent geometrical properties beginning with some simple metric and some good $1$-forms. Notice that the multiple $\b$-deformation is reversible general. Hence, for any Riemannian metric of excellent geometrical properties, possibly exist some $1$-forms to deform the metric to be a simple one.

Hence, we propose the following question.

\begin{center}
	\begin{minipage}{0.86\textwidth}
		{\it Question: Given an Einstein metric $\ba$, are there some $1$-forms $\b_1$, $\b_2$, $\cdots$, $\b_m$, such that $\ba$ is the resulting metric under multiple $\b$-deformations globally or locally, in which $\a$ is a Riemannian metric with constant sectional curvature?}
	\end{minipage}
\end{center}

There is some relationship between multiple $\b$-deformations and warped product metrics, which are given in the form as below,
\begin{eqnarray*}
	\bar g=\psi\ud r^2+\phi_1\sigma_1^2+\phi_2\sigma_2^2+\cdots+\phi_{n-1}\sigma_{n-1}^2.
\end{eqnarray*}
If one aims to find some metric satisfying some particular conditions, one can try to determine the suitable factors $\psi$ and $\phi_i$ above. As a classical method in differential geometric, it is simple and computable. It is easy to see that if $\psi$ and $\phi_i$ depend only on the functions $\langle\sigma_i,\sigma_j\rangle_g$ where $g$ is another simper metric such as
\begin{eqnarray*}
	g=\ud r^2+\sigma_1^2+\sigma_2^2+\cdots+\sigma_{n-1}^2,
\end{eqnarray*}
then $\bar g$ can be regarded as a resulting metric under multiple $\b$-deformations by the metric $g$ and the $1$-forms $\sigma_1$, $\sigma_2$, $\cdots$, $\sigma_{n-1}$.

For example, the Euclidean metric can be expressed as a warped product,
\begin{eqnarray*}
	g=\ud r^2+r^2(\sigma_1^2+\sigma_2^2+\cdots+\sigma_{n-1}^2),
\end{eqnarray*}
where all the $1$-forms $\sigma_i$ have constant length $1$ and orthogonal to any other ones. Denote $\alpha=\sqrt{g}$ and $\b_i=r^2\sigma_i$. Then $\langle\b_i,\b_j\rangle_\a=\delta_{ij}r^2$ in which $\delta_{ij}$ is the Kronecker symbol. In this case, the resulting metrics under multiple $\b$-deformations are given by
\begin{eqnarray*}
	\bar g=\psi(r^2)\ud r^2+\sum\phi_{ij}(r^2)\sigma_i\sigma_j,
\end{eqnarray*}
or particularly given by
\begin{eqnarray*}
	\bar g=\psi(r^2)\ud r^2+\phi_1(r^2)\sigma_1^2+\phi_2(r^2)\sigma_2^2+\cdots+\phi_{n-1}(r^2)\sigma_{n-1}^2
\end{eqnarray*}
when all the $\phi_{ij}(i\neq j)$ vanish.

If we use only one $1$-form $\b$, $\b_1$ for instant, to deform $\a$, then the resulting metrics are given by
\begin{eqnarray}
	\bar g=\psi(r^2)\ud r^2+\phi(r^2)\sigma_1^2+\varphi(r^2)(\sigma_2^2+\cdots+\sigma_{n-1}^2).\label{domainnaidnanng}
\end{eqnarray}

Hence, if we begin with the Euclidean metric and such particular $1$-form, then our problem is equivalent to determine all the Einstein metrics given in the form (\ref{domainnaidnanng}). Actually, all the examples constructed in this essay are in this form locally.

However, it does not mean that the method of multiple $\b$-deformations is just a new ``face" of warped product metrics. It is clear that we don't need to assume $\a$ must be a warped product metric. In the world of multiple $\b$-deformations, we only need two objects, a metric $\a$ and several $1$-forms $\b_i$.

The idea of multiple $\b$-deformations is very simple. However, the computation is a little bit complicated. In this essay, we just discuss $\b$-deformation, that is to say, only one $1$-form will be used. Even so, the crucial formula, namely Proposition \ref{RIkunderdeformaion} in Appendix, is obviously formidable.

Our main purpose is to study the structure of Einstein metrics under $\b$-deformations. By the way, we hope to construct some new Einstein metrics by the known ones. More specifically, assume $\a$ is an Einstein metric and $\b$ is a $1$-form, we want to understand the phenomenon of $\a$ and $\b$ under $\b$-deformations.

There are many choices on $\b$, and we believe that Killing forms, namely the dually $1$-forms of Killing vector fields, are a great choice. However, Theorem \ref{dinqienginadng} shows that it is impossible to obtain any non-trivial Einstein metric if there is not any restriction on the Killing form $\b$. Hence, we will add some special restrictions on the Killing form $\b$. There are many restrictions can be chosen. As the first essay of this subject, we will assume the Killing form satisfies a natural additional condition. Perhaps such condition is the simplest one, at least for even dimensional spaces.

This essay is organized in the following way.

Section \ref{idmiaindinngngng} introduces some necessary abbreviations. Section \ref{idimaihdhnaug} includes the definition and some basis facts of $\b$-deformations, partial contents are tedious, which are listed in Appendix.

In Section \ref{diamdinaihgng}, the necessary and sufficient condition of $\ba$ to be an Einstein metric, with $\a$ being an Einstein and $\b$ satisfying the crucial condition (\ref{owemigneigaindn}), is provided. By solving the related equations, we find that there is only one non-trivial solution, which is Ricci-flat and can only happen in $4$-dimensional case. In particular, when $\a$ is an Euclidean metric, such unique solution is Hawking Taub-NUT Riemannian metric.

The result is scarce if there is only one condition is added. However, we find that the constraint on the deformation factors will be relaxed by adding one more condition. In Section \ref{diamidnngiidning}, another condition (\ref{tuabgengand}) on $\b$ will be added. Such condition is natural, because it holds automatically when $\a$ is of constant sectional curvature. Theorem \ref{uenuangandgng} and Lemma \ref{donienannggabdg} show that, in order to make sure $\ba$ being an Einstein metric, the deformation factors $\kappa$ and $\rho$ only need to satisfy two ODEs.

The solutions of these two ODEs with the restriction $\kappa=0$ are special, which are provided in Lemma \ref{wyeyyabndgagd}. Actually, the corresponding deformations generate a transformation group, and all the data $\ab$ satisfying the additional conditions (\ref{owemigneigaindna})(namely (\ref{owemigneigaindn})) and (\ref{tuabgengand}) are closed under the action of this group. In particular, such group acts transitively on the subset in which $\a$ has constant sectional curvature. Moreover, the related discussion shows that our crucial condition  (\ref{owemigneigaindn}) only happens on even dimensional spaces. These contents are stated in Section \ref{di4ufnang}.

In Section \ref{dimiunuang}, by solving the crucial ODEs in Lemma \ref{donienannggabdg}, we  obtain all the possible deformation factors expressed as a form of integral. Not all the solutions are elementary. However, we find fortunately that if $\a$ is of constant sectional curvature, then we can express all the Einstein metrics as a warped product form (\ref{domainnaidnanng}) explicitly.

Section \ref{sinainhhdngnnag} dedicates to the $4$-dimensional case. By which one can see that our metrics cover many well-known $4$-dimensional Einstein metrics, such as Fubini-Study metric, Bergmann metric, Hawking Taub-NUT metric, Eguchi-Hanson's metrics and Pedersen's metrics, etc. They appear in various forms in many literatures. However, our can't cover all the Belinskii-Gibbons-Pope -Page metrics\cite{BGPP}
\begin{eqnarray*}
	g=\omega^{-1}\ud r^2
	+\omega r^2\left\{\left(1-\frac{m_1}{r^4}\right)^{-1}\sigma_1^2
	+\left(1-\frac{m_2}{r^4}\right)^{-1}\sigma_2^2
	+\left(1-\frac{m_3}{r^4}\right)^{-1}\sigma_3^2\right\}
\end{eqnarray*}
in which $\omega:=\sqrt{(1-\frac{m_1}{r^4})(1-\frac{m_2}{r^4})(1-\frac{m_3}{r^4})}$, unless $(m_1-m_2)(m_2-m_3)(m_3-m_1)=0$, which belong to Eguchi-Hanson's metrics actually. Obviously, the complete BGPP metrics can be covered by multiple $\b$-deformations.

As we have pointed out before, the initial metric can not necessarily be expressed as a warped product. In Section \ref{adiannnndnnggbabbggg}, we would like to express the resulting Einstein metrics by $\a$ and $\b$ directly. The advantages of doing so are expressing the metric $\a$ independent of coordinate systems, and meanwhile, showing the effect of the $1$-form $\b$ clearly.

It is not easy, since the deformation factors are hard to be expressed analytically. We can only provide serval examples. As the simplest one, the first example is great, which can be regarded as the generalization of Eguchi-Hanson metrics on higher dimensional spaces. In the second example, we provide a simple expression for Fubini-Study metric or Bergmann metric as $\sqrt{\a^2-\mu\b^2}$, which has Ricci constant $\bar\mu=\frac{n+2}{n-1}\mu$. Other interesting metrics are discussed in Example \ref{dubaudvajdvgjkhbag}-Example \ref{dubavffaccddhgga}.

We should point out that the completeness of the Einstein metrics constructed in this essay will not be argument. Parts of them are complete or can be extended to be a complete metric.

On the other hand, perhaps the resulting metric after $\b$-deformation will not defined on the original manifold because of some latent singularity. Fubini-Study metric is a typical example. Consider a standard rough sphere $\mathbb S^{2n}$ and its $1$-form satisfying (\ref{owemigneigaindn}), then the metric $\sqrt{\a^2-\b^2}$ defines on $\mathbb{CP}^n$ other than $\mathbb S^{2n}$, since it degenerates along the equator. See Example \ref{doaieningangga} for details.

Finally, it is worth mentioning that the method of $\b$-deformations is appropriate for any other pseudo-Riemannian metrics, although we focus on the Riemannian case.

\section{Preliminaries}\label{idmiaindinngngng}
In this essay, we need some abbreviations, which are frequently-used in the literatures on Finsler geometry.

Let $\a=\sqrt{a_{ij}(x)y^iy^j}$ be a Riemannian metric and $\b=b_iy^i$ be a $1$-form, in which $(x,y)$ is the natural local coordinate system for a $n$-dimensional manifold. Denote
\begin{eqnarray*}
	\rij:=\f{1}{2}(\bij+b_{j|i}),\quad\sij:=\f{1}{2}(\bij-b_{j|i})
\end{eqnarray*}
be the symmetrization and antisymmetrization of the covariant derivative $\bij$ respectively, then we can define some related tensors as below:
\begin{eqnarray*}
	&r_{i0}:=r_{ij}y^j,\quad r^i{}_0:=a^{im}r_{m0},\quad r_{00}:=r_{i0}y^i,\quad r_i:=b^mr_{mi},\quad r^i:=a^{im}r_{m},\quad r_0:=r_iy^i,\\
	&r:=r_ib^i,\quad s_{i0}:=s_{ij}y^j,\quad s^i{}_0:=a^{im}s_{m0},\quad s_i:=b^ms_{mi},\quad s^i:=a^{im}s_{m},\quad s_0:=s_iy^i.
\end{eqnarray*}
Roughly speaking, indices are raised or lowered by $a^{ij}$ or $\aij$, vanished by contracted with $b^i$~(or~$\bi$) and changed to be '${}_0$' by contracted with $y^i$ (or $y_i:=a_{ij}y^j$).

At the same time, we need other three tensors
\begin{eqnarray*}
	p_{ij}:=r_{im}r^m{}_j,\qquad q_{ij}:=r_{im}s^m{}_j,\qquad t_{ij}:=s_{im}s^m{}_j
\end{eqnarray*}
and the related tensors determined by the above rules. Notice that both $p_{ij}$ and $t_{ij}$ are symmetric, but $q_{ij}$
is neither symmetric nor antisymmetric in general. So $b^jq_{ji}$ and $b^jq_{ij}$, denoted by $q_i$ and $q^\star_i$ respectively, are different. But $b^iq_i$, denoted by $q$, is equal to $b^iq^\star_i$.

Moreover, in order to avoid ambiguity, sometimes we have to use index $b$, which means contracting the corresponding index with $b^i$ or $b_i$. For example, $\sohb:=s_{0|k}b^k$.

Finally, the coefficients of the Riemann tensor are denoted by $\RIk$, which are quadratic homogeneous on $y$. $\a$ has constant sectional curvature $\mu$ means $\RIk=\mu(\a^2\dIk-\yI\yk)$. A special quantity $\Rbb$, denoted by $\RIk\bi\bK$ according to our rules, is required. The Ricci tensor is denoted by $Ric_{ij}:=\frac{1}{2}[R^k{}_k]_{y_iy_j}$, with some related quantities $\Ricoo:=Ric_{ij}y^iy^j$, $\Rico:=Ric_{ij}y^ib^j$, and $\Ric:=R_{ij}b^ib^j$. $\a$ has constant Ricci curvature $\mu$ means $\Ricoo=(n-1)\mu\a^2$.

\section{$\b$-deformation and some basic formulea}\label{idimaihdhnaug}
Given a Riemannian metric $\a$ and a $1$-form $\b$. Denote $b:=\|\b\|_{\a}$. $\b$-deformation is a special metrical deformation as below
\begin{eqnarray*}
	\ba=e^{\rho(b^2)}\sqrt{\a^2-\kappa(b^2)\b^2},\qquad\bb=\nu(b^2)\b,
\end{eqnarray*}
including stretch and conformal deformation for the Riemannian metric and length deformation for the $1$-form.
\begin{figure}[h]
	\centering
	\includegraphics[scale=0.6]{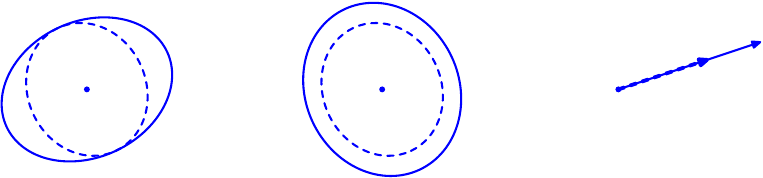}
	\caption{Stretch, Conformal and Length Deformation}
\end{figure}

Since $\baij=e^{2\rho}(\aij-\kappa\bi\bj)$, $\bbi=\nu\bi$, we have
\begin{eqnarray*}
	\bar a^{ij}=e^{-2\rho}\left(a^{ij}+\frac{\kappa}{1-b^2\kappa}b^ib^j\right),\qquad \bar b^i=\frac{e^{-2\rho}\nu}{1-b^2\kappa}b^i.
\end{eqnarray*}
As a result, the length of $\bb$ with respect to $\ba$ is determined by,
\begin{eqnarray*}
	\bar b^2=\frac{e^{-2\rho}\nu^2}{1-b^2\kappa}b^2.
\end{eqnarray*}

If we denote
\begin{eqnarray*}
	\Delta(b^2):=1-b^2\kappa,
\end{eqnarray*}
then $\ba$ can be determined by $\Delta$ and $\rho$ as below
\begin{eqnarray}\label{Drho}
	\ba={e^{\rho}\sqrt{\left(\a^2-\f{\b^2}{b^2}\right)+\Delta\f{\b^2}{b^2}}}.
\end{eqnarray}

Obviously that
\begin{eqnarray*}
	\ba^2-\f{\bb^2}{\bar b^2}=e^{2\rho}\left(\a^2-\f{\b^2}{b^2}\right),
\end{eqnarray*}
which indicates $\a^2-\frac{\b^2}{b^2}$ is conformal invariable under $\b$-deformations. Be attended that $\a^2-\frac{\b^2}{b^2}$ is positive semi-definite since it vanishes along the direction $\b^\sharp$. Hence, $\ba$ is a Riemannian metric if and only if $\Delta$ is positive due to (\ref{Drho}). By the way, if $\Delta$ is negative, one will obtain a Lorentzian metric with signature $(n-1,1)$. But $\Delta=0$ is forbidden.

Here are some basic facts about $\b$-deformations. Firstly, the composition of two $\b$-deformations is also a $\b$-deformation.
\begin{lemma}\label{bbchange}
	If $(\ba,\bb)$ is determined by $\b$-deformation with data $(\check\a,\check\b)$ and factors $(\check\Delta(\check b^2),\check\rho(\check b^2),\check\nu(\check b^2))$, and $(\check\a,\check\b)$ is determined by $\b$-deformation with data $(\a,\b)$ and factors $(\Delta(b^2),\rho(b^2),\nu(b^2))$, then $(\ba,\bb)$ can be determined by $\b$-deformation with data $(\a,\b)$ directly. The corresponding deformation factors are given by
	\begin{eqnarray*}
		\left(\check\Delta(\check b^2)\Delta(b^2),\check\rho(\check b^2)+\rho(b^2), \check\nu(\check b^2)\nu(b^2)\right).
	\end{eqnarray*}
\end{lemma}

\begin{figure}[h]
	\centering
	$\xymatrix{
		& (\check\alpha,\check\beta) \ar[dr]^{(\check\Delta,\check\rho,\check\nu)}             &\\
		(\alpha,\beta) \ar[ur]^{(\Delta,\rho,\nu)} \ar[rr]^{(\Delta\check\Delta,\rho+\check\rho,\nu\check\nu)} & &     (\ba,\bb)}$
	\caption{Composition of two $\b$-deformations}
\end{figure}
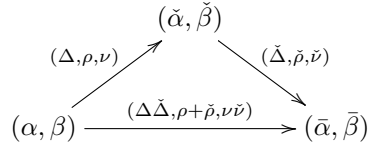

Next, there are three deformation formulae about non curvature tensors. The former two formulae are restatement of Lemma 2, 3, and 4 in \cite{yct-dhfp}, and the latter one is a direct inference of the former. More deformation formulae about curvature tensor $\RIk$, $\Rbb$ and $\Ricoo$ are listed in Appendix. All of them are very complicated. Any minor mistake will lead to disastrous results. In order to make sure the absolute accuracy, all the calculations are finished by Maple program and the results have been rigorously verified by using different known examples.

\begin{proposition}\label{relationunderdeformations}
	After $\b$-deformation,
	\begin{eqnarray*}
		\brij&=&\f{2\rho'\nu}{1-b^2\kappa}r(\aij-\kappa\bi\bj)-\f{\kappa'\nu}{1-b^2\kappa}r\bi\bj+\f{\nu}{1-b^2\kappa}\rij\\
		&&+\left(\f{b^2\kappa'\nu}{1-b^2\kappa}-2\rho'\nu+\nu'\right)\left\{\bi(\rj+\sj)+\bj(\ri+\si)\right\}
		+\f{\kappa\nu}{1-b^2\kappa}(\bi\sj+\bj\si).
	\end{eqnarray*}
\end{proposition}
\begin{remark}
	Notice that $\brij$ is the related tensor about the resulting data $(\ba,\bb)$. The same for the tensors $\bsij$, $\bar t_{ij}$, etc. below.
\end{remark}

\begin{proposition}\label{sounderdeformations}
	After $\b$-deformation,
	\begin{eqnarray*}
		\bsij&=&\nu\sij+\nu'\left\{\bi(\rj+\sj)-\bj(\ri+\si)\right\},\\
		\bsi&=&-\f{e^{-2\rho}\nu}{1-b^2\kappa}\left\{\nu'r\bi-b^2\nu'\ri-(\nu+b^2\nu')\si\right\}.
	\end{eqnarray*}
\end{proposition}

\begin{proposition}\label{toounderdeformaion}
	After $\b$-deformation,
	\begin{eqnarray*}
		\bar t_{00}&=&c_1\ro^2+c_2\ro\so+c_3\so^2+c_4\too+c_5\b\ro r+c_6\b\so r+c_7\b(\qo+\to)+c_8\b^2r^2\\
		&&+c_9\b^2(p-2q-t),
	\end{eqnarray*}
	where
	\begin{eqnarray*}
		&c_{1}=-\f{b^2e^{-2\rho}}{1-b^2\kappa}\nu'^2,\quad c_{2}=-\f{2e^{-2\rho}}{1-b^2\kappa}(\nu+b^2\nu')\nu',\quad
		c_{3}=-\f{e^{-2\rho}}{1-b^2\kappa}(\kappa\nu^2+2\nu\nu'+b^2\nu'^2),\\
		&c_{4}=e^{-2\rho}\nu^2,\quad c_{5}=\f{2e^{-2\rho}}{1-b^2\kappa}\nu'^2,\quad
		c_{6}=\f{2e^{-2\rho}}{1-b^2\kappa}(\kappa\nu+\nu')\nu',\\
		&c_{7}=2e^{-2\rho}\nu\nu',\quad
		c_{8}=-\f{\kappa e^{-2\rho}}{1-b^2\kappa}\nu'^2,\quad
		c_{9}=-e^{-2\rho}\nu'^2.
	\end{eqnarray*}
\end{proposition}

A $1$-form will be called a {\it Killing form} or {\it conformal form} of a given Riemannian metric if its dual vector field is a Killing or conformal vector field respectively. A conclusion about Killing forms under $\b$-deformations is provided below.

\begin{theorem}\label{KillingtoKillingu}
	Assume $\b$ is a Killing form of $\a$ with non-constant length. Then after $\b$-deformation, $\bb$ is conformal with respect to $\ba$ if and only if
	\begin{eqnarray*}
		\nu=k(1-b^2\kappa)e^{2\rho},
	\end{eqnarray*}
	where $k$ is a non-zero constant. Moreover, $\bb$ must be also a Killing form of $\ba$.
\end{theorem}
\begin{proof}
	Since $\b$ is a Killing form, by Proposition \ref{relationunderdeformations} we have
	\begin{eqnarray*}
		\brij=\left\{\f{\nu(\kappa+b^2\kappa')}{1-b^2\kappa}-2\nu\rho'+\nu'\right\}(\bi\sj+\bj\si).
	\end{eqnarray*}
	It is easy to see that $(b^2)_{|i}=2(\ri+\si)$, which indicates that $\so\neq0$ since $\b$ has non-constant length. Hence, $\bb$ is a Killing form, namely $\brij=0$, if and only if
	\begin{eqnarray*}
		\f{\nu(\kappa+b^2\kappa')}{1-b^2\kappa}-2\nu\rho'+\nu'=0,
	\end{eqnarray*}
	by which we have $\nu=k(1-b^2\kappa)e^{2\rho}$ for some constant $k$.\qed
\end{proof}

\section{Deformation with one additional condition on the Killing form}\label{diamdinaihgng}
Due to Ricci identity, we have a formula between the curvature tensor and some non curvature tensors about the covariant derivation of an arbitrary $1$-form as below,
\begin{eqnarray}\label{relation1}
	s_{ij|k}=-b^m R_{kmij}+r_{ik|j}-r_{jk|i},
\end{eqnarray}
Such formula is priori and plays a crucial role in the research of classifying Randers metrics with constant flag curvature\cite{db-robl-orso,db-robl-szm-zerm}(see also \cite{xy} for a summarization and simplification).

If $\b$ is a Killing form, (\ref{relation1}) implies the following facts,
\begin{eqnarray}
	\sohb=-t_{0},\qquad\sIohi=\Rico,\qquad\sIhi=-\Ric-\tIi.\label{relation}
\end{eqnarray}
As a reslut, combining (\ref{relation}) with Proposition \ref{Ricoounderdeformation} in Appendix, we can understand the behavior of Ricci curvature under $\b$-deformations when $\b$ is a Killing form.
\begin{lemma}\label{RicoounderdeformationKilling}
	Assume $\b$ is a Killing form of $\a$. Then after $\b$-deformation,
	\begin{eqnarray*}
		\bRicoo&=&\big\{C_1(\Ric+\tIi)+C_2 t\big\}(\a^2-\kappa\b^2)+(C_3\Ric+C_4\tIi+C_5 t)\b^2+C_6\so^2\\
		&&+C_7\too+C_8\b\to+C_9\soho+\Ricoo+C_{10}\b\Rico,
	\end{eqnarray*}
	where
	\begin{eqnarray*}
		C_1&=&\textstyle -\f{1}{b^2}+\f{1}{b^2}(1+2b^2\rho'),\nonumber\\
		C_2&=&\textstyle \f{n}{b^4}-\f{\Delta'}{b^2\Delta}-\f{1}{b^4\Delta}\big\{2(n-1)\Delta-b^2\Delta'\big\}(1+2b^2\rho')+\f{n-2}{b^4}(1+2b^2\rho')^2+\f{2}{b^2}(1+2b^2\rho')',\nonumber\\
		C_3&=&\textstyle -\f{1}{b^4}(\Delta-1-b^2\Delta'),\nonumber\\
		C_4&=&\textstyle -\f{1}{b^4}(\Delta^2-\Delta-b^2\Delta'),\nonumber\\
		C_5&=&\textstyle \f{1}{b^6}(\Delta-1)(2\Delta+n)-\f{1}{b^4\Delta}\big\{(n+3)\Delta-1\big\}\Delta'
		-\f{\Delta'^2}{b^2\Delta}+\f{2\Delta''}{b^2}-\f{n-2}{b^6}(\Delta-1-b^2\Delta')\\
		&&\textstyle\cdot(1+2b^2\rho'),\nonumber\\
		C_6&=&\textstyle -\f{n}{b^4}+\f{2\Delta}{b^4}-\f{2\Delta'}{b^2\Delta}+\f{\Delta'^2}{\Delta^2}-\f{2\Delta''}{\Delta}
		+\f{n-2}{b^4}(1+2b^2\rho')^2-\f{2(n-2)}{b^2}(1+2b^2\rho')',\nonumber\\
		C_7&=&\textstyle \f{2}{b^2}(\Delta-1),\nonumber\\
		C_8&=&\textstyle -\f{2n}{b^4}(\Delta-1)+\f{2(3\Delta-1)\Delta'}{b^2\Delta}+\f{2(n-2)}{b^4}(\Delta-1)(1+2b^2\rho'),\nonumber\\
		C_9&=&\textstyle \f{n-2}{b^2}-\f{\Delta'}{\Delta}-\f{n-2}{b^2}(1+2b^2\rho'),\nonumber\\
		C_{10}&=&\textstyle \f{2}{b^2}(\Delta-1).
	\end{eqnarray*}
\end{lemma}

Assume $\a$ has constant Ricci curvature $\mu$, then
\begin{eqnarray*}
	\Ricoo=(n-1)\mu\a^2,\qquad\Rico=(n-1)\mu\b,\qquad\Ric=(n-1)\mu b^2.
\end{eqnarray*}
After $\b$-deformation, the metric $\ba$ also has constant Ricci curvature $\bar\mu$ if and only if $\bRicoo=(n-1)\bar\mu\ba^2$. Combining with Lemma \ref{RicoounderdeformationKilling}, we immediately have
\begin{theorem}\label{dinqienginadng}
	Assume $\a$ is an Einstein metric with Ricci constant $\mu$, $\b$ is a Killing form of $\a$. Then after $\b$-deformation, $\ba$ is an Einstein metric with Ricci constant $\bar\mu$ if and only if
	\begin{eqnarray}\label{ybaeugald}
		E_1\a^2+E_2\b^2+E_3\so^2+E_4\too+E_5\b\to+E_6\soho=0,
	\end{eqnarray}
	where
	\begin{eqnarray*}
		E_1&=&\textstyle 2\rho'\tIi+2\left\{\f{\Delta'}{\Delta}\rho'+2(n-2)\rho'^2+2\rho''\right\}t+(n-1)\mu(1+2b^2\rho')-(n-1)\bar\mu e^{2\rho},\nonumber\\
		E_2&=&\textstyle -\f{1}{b^4}\left\{\Delta^2-1-b^2\Delta'-(\Delta-1)(1+2b^2\rho')\right\}\tIi
		+\f{1}{b^6\Delta}\big\{2\Delta(\Delta^2-1)-b^2(5\Delta\nonumber\\
		&&\textstyle -1)\Delta'-b^4\Delta'^2+2b^4\Delta\Delta''-2(n-2)b^2\Delta(\Delta-1)\rho'
		+2b^4[(n-1)\Delta-1]\Delta'\rho'\nonumber\\
		&&\textstyle +4(n-2)b^4\Delta(\Delta-1)\rho'^2+4b^4\Delta(\Delta-1)\rho''\big\}t+\f{n-1}{b^2}\mu\big\{b^2\Delta'+(\Delta-1)(1\nonumber\\
		&&\textstyle +2b^2\rho')\big\}-\f{n-1}{b^2}\bar\mu(\Delta-1)e^{2\rho},\nonumber\\
		E_3&=&\textstyle \f{2}{b^4}(\Delta-1)-\f{2\Delta'}{b^2\Delta}+\f{\Delta'^2}{\Delta^2}-\f{2}{\Delta}\Delta''+4(n-2)(\rho'^2-\rho''),\nonumber\\
		E_4&=&\textstyle \f{2}{b^2}(\Delta-1),\nonumber\\
		E_5&=&\textstyle -\f{4}{b^4}(\Delta-1)+\f{2}{b^2\Delta}(3\Delta-1)\Delta'+\f{4}{b^2}(n-2)(\Delta-1)\rho',\nonumber\\
		E_6&=&\textstyle -\f{\Delta'}{\Delta}-2(n-2)\rho'.
	\end{eqnarray*}
\end{theorem}

If there is not any inner relationship between the tensors $\a^2$, $\b^2$, $\so$, $\too$, $\to$ and $\soho$, then (\ref{ybaeugald}) holds if and only if all the coefficients $E_i$ vanish, which implies $\kappa=0$ and $\rho$ is a constant. That is to say, $\ba$ is just a re-scaling of $\a$. It is trivial.

However, such tensors may be of some inner relationship for some particular Killing forms. For example, denote $x:=(x_1,x_2,\cdots,x_{2m})^T$, $y:=(y_1,y_2,\cdots,y_{2m})^T$, and a $2m\times 2m$ symplectic matrix $J$ as below
\begin{eqnarray*}
	J:=\begin{bmatrix}
		0 & 1 &  &  &  \\
		-1 & 0 &  &  &  \\
		&  & \ddots &  &  \\
		&  &  & 0 & 1 \\
		&  &  & -1 & 0
	\end{bmatrix},
\end{eqnarray*}
then
\begin{eqnarray}
	\a=|y|,\qquad \b=\langle Jx,y\rangle\label{inbadbgamgd}
\end{eqnarray}
satisfies
\begin{eqnarray*}
	\too=-\a^2.
\end{eqnarray*}

More generally, take
\begin{eqnarray}
	\a&=&\f{2|y|}{1+\mu|x|^2},\qquad\b=\f{4a\langle Jx,y\rangle}{(1+\mu|x|^2)^2}\label{inbadbgamgddanda},
\end{eqnarray}
where $a$ is a non-zero constant, then $\a$ has constant sectional curvature $\mu$ and $\b$ is a Killing form of $\a$ satisfying a special equality as below
\begin{eqnarray}
	\too+(a^2-\mu b^2)\a^2+\mu(b^2\a^2-\b^2)-\f{\mu}{a^2-\mu b^2}\so^2=0.\label{owemigneigaindn}
\end{eqnarray}
It is worth to remark that (\ref{inbadbgamgddanda}) can be reexpressed in a projective coordinate system as
\begin{eqnarray}
	\a&=&\f{\sqrt{(1+\mu|x|^2)|y|^2-\mu\langle x,y\rangle^2}}{1+\mu|x|^2},\qquad\b=\f{a\langle Jx,y\rangle}{1+\mu|x|^2}\label{inbadbgamgddandab},
\end{eqnarray}

\begin{remark}
	Be attend that $a^2-\mu b^2$ in (\ref{owemigneigaindn}) may be equal to $0$ at some points, but $\so$ is also equal to $0$ in this case. In particular, if $\a$ and $\b$ are given by ()\ref{inbadbgamgddanda}, then
	\begin{eqnarray*}
		a^2-\mu b^2=\f{a^2(1-\mu|x|^2)^2}{(1+\mu|x|^2)^2},\qquad\so=\f{4a^2(1-\mu|x|^2)\langle x,y\rangle}{(1+\mu|x|^2)^3}.
	\end{eqnarray*}
\end{remark}
When $\mu>0$, both $a^2-\mu b^2$ and $\so$ are equal to $0$ if $1-\mu|x|^2=0$(namely the equator on the sphere). In this case, the item $\f{\so^2}{a^2-\mu b^2}$ should be regard as its limit, i.e.,
\begin{eqnarray*}
	\f{\so^2}{a^2-\mu b^2}=\f{16a^2\langle x,y\rangle^2}{(1+\mu|x|^2)^4},
\end{eqnarray*}
which is well-defined on the whole sphere.

There are two facts should be pointed out before going any further.

Firstly, we will only consider Killing form with non-constant length, because the case when $\b$ being of constant length is trivial. The reason is given below.

Since $(b^2)_{|k}=2(r_k+s_k)$, the Killing form $\b$ has constant length if and only if $\so=0$. So the condition (\ref{owemigneigaindn}) indicates that $\to=-(a^2-\mu b^2)\b$. On the other hand, $\to=0$ due to its definition~(see Section \ref{idmiaindinngngng}). So $a^2-\mu b^2\equiv0$. As a result, $\too=-\mu(b^2\a^2-\b^2)$. As a quadratic form, $\too$ must be negative definite or negative semi-definite. So the above equality implies $\mu\geq0$. When $\mu=0$, $\too=0$, which is equivalent to $\sij=0$. In this case, $\b$ is parallel and hence needed not to be discussed. When $\mu>0$, we have $t=0$ and $\tIi=-(n-1)\mu b^2$. According to Theorem \ref{dinqienginadng}, $\ba$ has constant Ricci curvature $\bar\mu$ if and only if
\begin{eqnarray*}
	&\left\{(n+1-2\Delta)\mu-(n-1)e^{2\rho}\bar\mu\right\}b^2\a^2\\
	&+(\Delta-1)\left\{[(n+1)+(n-1)\Delta]\mu-(n-1)e^{2\rho}\bar\mu\right\}\b^2=0.
\end{eqnarray*}
The only suitable solution of the above equation is $\Delta=1$ and $\bar\mu=e^{-2\rho}\mu$ in which $\rho$ is a constant. That is to say, $\ba$ is just a re-scaling of $\a$.

Secondly, according to Remark \ref{dinnaindbagfffafdg} in Section \ref{di4ufnang}, the phenomenon that $\b$ is a Killing form with non-constant length and satisfies (\ref{owemigneigaindn}) can only happens on even dimensional manifolds.

\begin{theorem}\label{tebangengald}
	Assume $\a$ is an Einstein metric with Ricci constant $\mu$, $\b$ is a Killing form of $\a$ with non-constant length and satisfies  (\ref{owemigneigaindn}) in which $a\neq0$. Then after $\b$-deformation, $\ba$ is an Einstein metric with Ricci constant $\bar\mu$ if and only if
	\begin{eqnarray}\label{ydaingieng}
		E_1\a^2+E_2\b^2+E_3\so^2+E_4\soho=0,
	\end{eqnarray}
	where
	\begin{eqnarray}\label{yabebgmadd}
		E_1&=&\textstyle \f{1}{b^2\Delta}\Big\{2a^2\Delta(1-\Delta)+(n-2)\mu b^2\Delta+(a^2-\mu b^2)b^2\Delta'+[(n-2)a^2\Delta\nonumber\\
		&&\textstyle -(n-3)\mu b^2\Delta-(a^2-\mu b^2)b^2\Delta'](1+2b^2\rho')-(n-2)(a^2-\mu b^2)\Delta\nonumber\\
		&&\textstyle \cdot(1+2b^2\rho')^2-2(a^2-\mu b^2)b^2\Delta(1+2b^2\rho')'-(n-1)\bar\mu b^2\Delta e^{2\rho}\Big\},\nonumber\\
		E_2&=&\textstyle -\f{1}{b^4\Delta}\Big\{a^2\Delta(1-\Delta)[n+(n-2)\Delta]+(2a^2-3\mu b^2)b^2\Delta\Delta'\nonumber\\
		&&\textstyle -(a^2-\mu b^2)b^4\Delta'^2+2(a^2-\mu b^2)b^4\Delta\Delta''+\mu b^2\Delta(1-\Delta)\nonumber\\
		&&\textstyle \cdot(1+2b^2\rho')-(a^2-\mu b^2)b^2[1-(n-1)\Delta]\Delta'(1+2b^2\rho')\nonumber\\
		&&-(n-2)(a^2-\mu b^2)\Delta(1-\Delta)(1+2b^2\rho')^2-2(a^2-\mu b^2)\nonumber\\
		&&\textstyle \cdot b^2\Delta(1-\Delta)(1+2b^2\rho')'-(n-1)\bar\mu b^2\Delta(1-\Delta)e^{2\rho}\Big\},\nonumber\\
		E_3&=&\textstyle -\f{1}{(a^2-\mu b^2)b^4\Delta^2}\Big\{a^2\Delta^2(n-2\Delta)-(n-2)\mu b^2\Delta^2+2(a^2-\mu b^2)\nonumber\\
		&&\textstyle \cdot b^2\Delta\Delta'-(a^2-\mu b^2)b^4\Delta'^2+2(a^2-\mu b^2)b^4\Delta\Delta''-(n-2)\nonumber\\
		&&\textstyle \cdot(a^2-\mu b^2)\Delta^2(1+2b^2\rho')^2+2(n-2)(a^2-\mu b^2)b^2\Delta^2(1+2b^2\rho')'\Big\},\nonumber\\
		E_4&=&\textstyle -\f{\Delta'}{\Delta}-2(n-2)\rho'.
	\end{eqnarray}
\end{theorem}
\begin{proof}
	Under the assumption (\ref{owemigneigaindn}), we have $\to=-(a^2-\mu b^2)\b$, $t=-(a^2-\mu b^2)b^2$ and $\tIi=-na^2+2\mu b^2$. Plugging them into (\ref{ybaeugald}) yields the result.\qed
\end{proof}

\begin{lemma}\label{cingoangnag}
	When $n>2$, all the suitable solutions of Eqs. (\ref{yabebgmadd}) are
	\begin{eqnarray}
		\Delta=1,\quad\rho=\mathrm{const.},\quad\bar\mu=e^{-2\rho}\mu\label{doannbbabdgad}
	\end{eqnarray}
	for any integer $n$, and
	\begin{eqnarray}
		\Delta(b^2)=\f{C^2}{(C+Db^2)^2},\quad\rho(b^2)=\f{1}{2}\ln(C+Db^2),\quad\mu=\bar\mu=0\label{aonandnangnandnbg}
	\end{eqnarray}
	only for $n=4$, where $C$ and $D$ are constants.
\end{lemma}
\begin{proof}
	Solving $E_4=0$ yields
	\begin{eqnarray}\label{daobabdbbbbbbdga}
		\Delta(b^2)=Ee^{-2(n-2)\rho(b^2)}
	\end{eqnarray}
	with $E$ being a constant, which indicates that $\Delta$ is a constant if and only if $\rho$ is a constant. If $\Delta$ is a constant,  $E_3=0$ reads $\frac{2a^2(\Delta-1)}{b^2(a^2-\mu b^2
		)}=0$. So $\Delta=1$. In this case, $E_2=0$ holds automatically and $E_1=0$ reads $(n-1)(\mu-e^{2\rho}\bar\mu)=0$. So $\bar\mu=e^{-2\rho}\mu$. The corresponding solution is (\ref{doannbbabdgad}).
	
	It is easy to verify, combining with the condition (\ref{daobabdbbbbbbdga}), that
	\begin{eqnarray*}
		&&2(1-\Delta)E_1+2b^2E_2-n(a^2-\mu b^2)b^2\Delta E_3\\
		&=&4(n-2)b^2\Delta\left\{(a^2-\mu b^2)[n(n-3)\rho'^2+2\rho'']-\mu\rho'\right\}.
	\end{eqnarray*}
	Solving the equation $(a^2-\mu b^2)[n(n-3)\rho'^2+2\rho'']-\mu\rho'=0$ yields
	\begin{eqnarray*}
		\rho(b^2)=\left\{ \begin{aligned}
			&\textstyle C+Db^2,\qquad&\mu=0,\\
			&\textstyle C+D\sqrt{a^2-\mu b^2},\quad&\mu\neq0,
		\end{aligned} \right.
	\end{eqnarray*}
	when $n=3$, or
	\begin{eqnarray*}
		\rho(b^2)=\left\{ \begin{aligned}
			&\textstyle\frac{2\ln(C+Db^2)}{n(n-3)},\qquad&\mu=0,\\
			&\textstyle\frac{2\ln(C+D\sqrt{a^2-\mu b^2})}{n(n-3)},\quad&\mu\neq0,
		\end{aligned} \right.
	\end{eqnarray*}
	when $n\geq4$, in which $C$ and $D$ are constants.
	
	For the case $n=3$, $E_3=0$ implies $D=0$ and $E=e^{2C}$. The corresponding solution is (\ref{doannbbabdgad}).
	
	For the case $n\geq4$ and $\mu=0$, $E_3=0$ is equivalent to
	\begin{eqnarray}\label{doainahhdbbbabggg}
		&&(n-4)(n^2-3n+4)D^2b^4+2n(n-1)(n-4)CDb^2+n^2(n-3)C^2\nonumber\\
		&=&n^2(n-3)E(C+Db^2)^{2-\frac{4(n-2)}{n(n-3)}}.
	\end{eqnarray}
	Because $b$ is not a constant, the above equation holds only if $\frac{4(n-2)}{n(n-3)}=0$, $1$ or $2$, by which we have $n=4$ and by (\ref{doainahhdbbbabggg}) we have $E=C^2$. Finally, $E_1=0$ and $E_2=0$ implies $\bar\mu=0$. The corresponding solution is (\ref{aonandnangnandnbg}).
	
	For the case $n\geq4$ and $\mu\neq0$, $E_3=0$ can be expressed as
	\begin{eqnarray*}
		[C+D\sqrt{a^2-\mu b^2}]^{-\frac{4(n-2)}{n(n-3)}}=\mathcal F(\sqrt{a^2-\mu b^2},b^2)
	\end{eqnarray*}
	where $\mathcal F$ is a two-variables rational function. Because $b$ is not a constant, the above equation holds only if either $D=0$ or $\frac{4(n-2)}{n(n-3)}$ is an integer. For the  former, $\rho$ is a constant and the corresponding solution is (\ref{doannbbabdgad}). For the latter, by the inequality $\frac{4(n-2)}{n(n-3)}\geq1$ we have $n=4$ or $5$. But it is only when $n=4$ that $\frac{4(n-2)}{n(n-3)}$ becomes an integer. Finally, $E_3=0$  yields $D=0$ and $E=C^2$, which indicates $\rho$ is a constant and the corresponding solution is (\ref{doannbbabdgad}).\qed
\end{proof}

(\ref{doannbbabdgad}) means $\ba$ is a re-scaling of $\a$. Excluding this trivial case, we immediately have
\begin{theorem}\label{aondbabdgmmadbgg}
	Assume $\a$ is an Einstein metric with Ricci constant $\mu$, $\b$ is a Killing form of $\a$ with non-constant length and satisfying  (\ref{owemigneigaindn}) in which $a\neq0$. If $\a^2$, $\b^2$, $\so^2$ and $\soho$ are linear independent, then after $\b$-deformation, $\ba$ is an Einstein metric (excluding the trivial case) if and only if $\a$ is a $4$-dimensional Ricci-flat metric and $\ba$ is determined by
	\begin{eqnarray}\label{Hawkingmetric}
		\ba=\f{\sqrt{(C+Db^2)^2(b^2\a^2-\b^2)+C^2\b^2}}{b\sqrt{C+Db^2}}
	\end{eqnarray}
	in which $C$ and $D$ are non-zero constants. In particular, $\ba$ is also Ricci-flat.
\end{theorem}

When $\a$ and $\b$ are given by (\ref{inbadbgamgd}), then (\ref{Hawkingmetric}) is the famous Hawking Taub-NUT Riemannian metric\cite{Hawking,LeBrun}. See Section \ref{sinainhhdngnnag} for the related discussions.

\section{Deformation with another additional condition on the Killing form}\label{diamidnngiidning}
In Theorem \ref{aondbabdgmmadbgg}, we assume the tensors $\a^2$, $\b^2$, $\so^2$ and $\soho$ are linear independent. In this case, there is only one non-trivial way to deform $\a$ to be a new Einstein metric. However, for a particular Einstein metric, it is possible to exist some more potential relationships between such tensors. If so, we can expect logically to obtain more Einstein metrics since it will relax the restrictions on deformation factors.

Actually, due to the priori formula  (\ref{relation1}), the tensor $\soho$ of a Killing form has a priori relationship with the Riemann curvature tensor as below,
\begin{eqnarray}
	\soho=-\Rbb-t_{00}.\label{relation4ad}
\end{eqnarray}
Hence, if we add a hypothesis that 
\begin{eqnarray*}
	\Rbb=\mu(b^2\a^2-\b^2),
\end{eqnarray*}
then (\ref{owemigneigaindn}) together with (\ref{relation4ad}) yields
\begin{eqnarray}\label{eyabgadjg}
	\soho=(a^2-\mu b^2)\a^2-\f{\mu}{a^2-\mu b^2}\so^2.
\end{eqnarray}
That is to say, the term $\soho$ in (\ref{ydaingieng}) can be erased by adding some further curvature condition!

The condition $\Rbb=\mu(b^2\a^2-\b^2)$ is the simplest one that we can expect, because it will hold automatically for any metric with constant sectional curvature. Actually, such condition indicates that the sectional curvature on any sectional plane including the vector $\b^\sharp$ is equal to $\mu$.

By the way, the whole discussions below is possible to be extended by adding some other suitable hypothesis on curvature tensor.

\begin{theorem}\label{uenuangandgng}
	Assume dimension $n>2$, $\a$ is an Einstein metric with Ricci constant $\mu$, and $\b$ is a Killing form of $\a$ with non-constant length and satisfying two additional conditions
	\begin{eqnarray}
		&\displaystyle\too+(a^2-\mu b^2)\a^2+\mu(b^2\a^2-\b^2)-\f{\mu}{a^2-\mu b^2}\so^2=0,&\label{owemigneigaindna}\\
		&\Rbb=\mu(b^2\a^2-\b^2),\label{tuabgengand}&
	\end{eqnarray}
	where $a\neq0$ is a constant. Then after $\b$-deformation, $\ba$ is an Einstein metric with Ricci constant $\bar\mu$ if and only if
	\begin{eqnarray*}
		X=Y=Z=0,
	\end{eqnarray*}
	where
	\begin{eqnarray}\label{yeangenga}
		X&=&\textstyle \f{1}{b^2\Delta}\Big\{a^2\Delta(n-2\Delta)+b^2[\mu\Delta-(a^2-\mu b^2)\Delta'](1+2b^2\rho')\nonumber\\
		&&\textstyle -(n-2)(a^2-\mu b^2)\Delta(1+2b^2\rho')^2-2(a^2-\mu b^2)b^2\nonumber\\
		&&\textstyle \cdot\Delta(1+2b^2\rho')'-(n-1)\bar\mu b^2\Delta e^{2\rho}\Big\},\nonumber\\
		Y&=&\textstyle -\f{1}{b^4\Delta}\Big\{a^2\Delta(1-\Delta)[n+(n-2)\Delta]+(2a^2-3\mu b^2)b^2\Delta\Delta'\nonumber\\
		&&\textstyle -(a^2-\mu b^2)b^4\Delta'^2+2(a^2-\mu b^2)b^4\Delta\Delta''+\mu b^2\Delta(1-\Delta)\nonumber\\
		&&\textstyle\cdot(1+2b^2\rho')-(a^2-\mu b^2)b^2[1-(n-1)\Delta]\Delta'(1+2b^2\rho')\nonumber\\
		&&\textstyle -(n-2)(a^2-\mu b^2)\Delta(1-\Delta)(1+2b^2\rho')^2-2(a^2-\mu b^2)b^2\Delta\nonumber\\
		&&\textstyle \cdot(1-\Delta)(1+2b^2\rho')'-(n-1)\bar\mu b^2\Delta(1-\Delta)e^{2\rho}\Big\},\nonumber\\
		Z&=&\textstyle -\f{1}{(a^2-\mu b^2)b^4\Delta^2}\Big\{a^2\Delta^2(n-2\Delta)+(2a^2-3\mu b^2)b^2\Delta\Delta'\nonumber\\
		&&\textstyle -(a^2-\mu b^2)b^4\Delta^2+2(a^2-\mu b^2)b^4\Delta\Delta''-(n-2)\mu b^2\Delta^2\nonumber\\
		&&\textstyle \cdot(1+2b^2\rho')-(n-2)(a^2-\mu b^2)\Delta^2(1+2b^2\rho')^2\nonumber\\
		&&\textstyle +2(n-2)(a^2-\mu b^2)b^2\Delta^2(1+2b^2\rho')'\Big\}.
	\end{eqnarray}
\end{theorem}
\begin{proof}
	Plugging (\ref{eyabgadjg}) into (\ref{ydaingieng}) yields
	\begin{eqnarray*}
		X\a^2+Y\b^2+Z\so^2=0.
	\end{eqnarray*}
	Since $n>2$, $\a^2$ as a quadratic form can not be expressed as a sum of square of two linear functions. So the above equality holds if and only if $X=0$ and $Y\b^2+Z\so^2=0$. Finally, notice that $\b$ ia orthogonal to $\so$, so $Y\b^2+Z\so^2=0$ if and only if $Y=Z=0$.\qed
\end{proof}

We don't need to add any more condition. Actually, it is impossible to do so, because the tensors $\a^2$, $\b^2$ and $\so^2$ cannot be linearly related unless $n=2$ due to above discussions. Hence, the next thing that we should do is to solve the equations and obtain the Einstein metrics.

\begin{lemma}\label{donienannggabdg}
	When $n>2$ and $1+2b^2\rho'\neq0$, Eqs. (\ref{yeangenga}) are equivalent to
	\begin{eqnarray*}
		X=0,\qquad T=0,
	\end{eqnarray*}
	where
	\begin{eqnarray}\label{einangnang}
		T&=&(a^2-\mu b^2)\big\{\Delta^2-b^2\Delta'(1+2b^2\rho')-\Delta(1+2b^2\rho')^2\nonumber\\
		&&+2b^2\Delta(1+2b^2\rho')'\big\}+\mu b^2\big\{\Delta^2-\Delta(1+2b^2\rho')\big\}.
	\end{eqnarray}
\end{lemma}
\begin{proof}
	$T=0$ holds since
	\begin{eqnarray}
		b^2(1-\Delta)X+b^4Y-(a^2-\mu b^2)b^4\Delta Z=(n-2)T.\label{oaniandbbgbbad}
	\end{eqnarray}
	
	Conversely, if $X=T=0$, then by the fact
	\begin{eqnarray*}
		b^2\Delta\left\{2\rho'X-X'+(a^2-\mu b^2)(1+2b^2\rho')Z\right\}-T'=0
	\end{eqnarray*}
	we have $Z=0$. By (\ref{oaniandbbgbbad}) we have $Y=0$.\qed
\end{proof}
\begin{remark}
	When $1+2b^2\rho'=0$, then $T=0$ (according to (\ref{oaniandbbgbbad}), $T=0$ holds if $X=Y=Z=0$ holds when $n>2$) implies $\Delta=0$, which is forbidden in the theory of $\b$-deformations since $\ba$ is not a Riemannian metric in this case.
\end{remark}

Next, we will try to solve the ODEs $X=0$ and $T=0$. Previously, there are two special cases: $\Delta=1$ or $\rho=\mathrm{const.}$. The former case means that we just use conformal deformations to deform $\a$, and latter means that we just use stretch deformations.

\begin{lemma}\label{wyeyyabndgagd}
	When $n>2$ and $\Delta=1$, the solutions of (\ref{yeangenga}) are given by
	\begin{eqnarray}
		\rho(b^2)=-\ln(C+Db^2),\qquad\bar\mu=4CDa^2\label{dinewbbabvavda}
	\end{eqnarray}
	when $\mu=0$, or
	\begin{eqnarray}
		\rho(b^2)=-\ln(C+D\sqrt{a^2-\mu b^2}),\qquad\bar\mu=(C^2-D^2a^2)\mu\label{dinewbbabvavdb}
	\end{eqnarray}
	when $\mu\neq0$.
\end{lemma}

\begin{remark}
	When $\a$ is of constant sectional curvature, then the corresponding metric determined by (\ref{dinewbbabvavda}) or (\ref{dinewbbabvavdb}) is also of  constant sectional curvature. See Theorem \ref{tebhagnagad} for reasons.
\end{remark}

\begin{lemma}\label{dianidnaingnandnd}
	When $n>2$ and $\rho=\mathrm{const.}$, the solutions of (\ref{yeangenga}) are given by
	\begin{eqnarray*}
		\Delta(b^2)=1,\qquad\bar\mu=e^{-2\rho}\mu
	\end{eqnarray*}
	or
	\begin{eqnarray}
		\Delta(b^2)=1-\frac{\mu}{a^2}b^2,\qquad\bar\mu=\frac{n+2}{n-1}e^{-2\rho}\mu.\label{dinewbbabvavdd}
	\end{eqnarray}
\end{lemma}

\begin{remark}
	When $\a$ is a metric with constant sectional curvature, then the corresponding metric determined by (\ref{dinewbbabvavdd}) is actually Fubini-Study metric when $\mu>0$ or Bergmann metric when $\mu<0$. See Example \ref{doaieningangga}  in Section \ref{adiannnndnnggbabbggg}  for details.
\end{remark}

\section{Some related discussions}\label{di4ufnang}
\begin{theorem}\label{doandnabbgbadfg}
	Suppose $\a$ and $\b$ satisfy all the assumptions in Theorem \ref{uenuangandgng}. Then after $\b$-deformation, $\ba$ has constant Ricci curvature $\bar\mu$, and meanwhile $\bb$ satisfies
	\begin{eqnarray}
		\bRbb=\bar\mu(\bar b^2\ba^2-\bb^2)\label{ianeingangngaagg}
	\end{eqnarray}
	if and only if $\Delta$ and $\rho$ are determined by Lemma \ref{wyeyyabndgagd}. Moreover, when $\nu=k(1-b^2\kappa)e^{2\rho}$, $\bb$ is a Killing form of $\ba$ and satisfies
	\begin{eqnarray}
		\btoo+(\bar a^2-\bar\mu \bar b^2)\ba^2+\bar\mu(\bar b^2\ba^2-\bb^2)
		-\f{\bar\mu}{\bar a^2-\bar\mu \bar b^2}\bso^2=0,\label{owemigneigaindnab}
	\end{eqnarray}
	in which $\bar a=\pm k a$.
\end{theorem}
\begin{proof}
	Clearly, the property (\ref{ianeingangngaagg}) does not depend on the choice of $\nu$. By Proposition \ref{Rbbunderdeformaion} in Appendix, (\ref{ianeingangngaagg}) is equivalent to
	\begin{eqnarray}\label{donaidnnngabbgg}
		P(b^2\a^2-\b^2)+\f{Q\Delta}{(a^2-\mu b^2)}\so^2=0,
	\end{eqnarray}
	where
	\begin{eqnarray*}
		P&=&a^2\Delta^2-(a^2-\mu b^2)\left\{b^2\Delta'+\Delta(1+2b^2\rho')\right\}(1+2b^2\rho')-\bar\mu b^2\Delta e^{2\rho},\\
		Q&=&(a^2-\mu b^2)b^4Z-3a^2(\Delta-1)+2(n-3)b^4\left\{2(a^2-\mu b^2)(\rho''-\rho'^2)-\mu\rho'\right\}.
	\end{eqnarray*}
	Notice that $b^2\a^2-\b^2$ as a matrix is of rank $n-1$ but $\so^2$ is of rank $1$, so (\ref{donaidnnngabbgg}) holds if and only if $P=Q=0$ since $n>2$. On the other hand, $X=Y=Z=0$ by Theorem \ref{uenuangandgng} and $T=0$ by Lemma \ref{donienannggabdg}.
	
	As a result, we conclude $\Delta=1$. It is true due to the following fact
	\begin{eqnarray*}
		&&3na^4\Delta^3(\Delta-1)\\
		&=&(n-3)a^2b^2\Delta^2X-na^2\Delta^3(a^2-\mu b^2)b^4Z+(n-3)\big\{(n+1)a^2\Delta^2+(a^2-\mu b^2)\\
		&&\cdot(1+2b^2\rho')[\Delta(1+2b^2\rho')+b^2\Delta']\big\}T+(n-3)\big\{(n-1)a^2\Delta^2+2(a^2-\mu b^2)\\
		&&\cdot b^2\Delta'(1+2b^2\rho')\big\}P+2(n-3)(a^2-\mu b^2)b^2\Delta(1+2b^2\rho')P'-\Delta^2\big\{na^2\Delta\\
		&&+(n-3)(a^2-\mu b^2)(1+2b^2\rho')^2\big\}Q.
	\end{eqnarray*}
	
	Hence, the necessity holds by Lemma \ref{wyeyyabndgagd}. And the sufficiency holds by verifying that $\Delta$ and $\rho$ satisfy $P=Q=0$.
	
	Finally, according to Theorem \ref{KillingtoKillingu}, $\bb$ becomes a Killing form if and only if $\nu=k(1-b^2\kappa)e^{2\rho}$. In this case, (\ref{owemigneigaindnab}) holds due to Proposition \ref{sounderdeformations} and Proposition \ref{toounderdeformaion}.\qed
\end{proof}

\begin{theorem}\label{tebhagnagad}
	Assume dimension $n>2$, $\a$ has constant sectional curvature $\mu$, and $\b$ is a Killing form of $\a$ with non-constant length and satisfying  (\ref{owemigneigaindn}) in which $a\neq0$. Then after $\b$-deformation, $\ba$ has constant sectional curvature $\bar\mu$ if and only if $\Delta$ and $\rho$ are determined by Lemma \ref{wyeyyabndgagd}. Moreover, when $\nu=k(1-b^2\kappa)e^{2\rho}$, $\bb$ is a Killing form of $\ba$ satisfying (\ref{owemigneigaindnab}) in which $\bar a=\pm k a$.
\end{theorem}
\begin{proof}
	Sufficiency:~Combining the given data $\Delta$, $\rho$ and $\bar\mu$ with the conditions on $(\a,\b)$, by Proposition \ref{RIkunderdeformaion} in Appendix we have $\bRIk=\bar\mu(\ba^2\dIk-\yI\bar y_k)$ in which $\bar y_k:=\bar a_{kl}y^l$.
	
	Necessity:~It is true due to Theorem \ref{doandnabbgbadfg}.\qed
\end{proof}

Consider the set, denoted by $\Lambda$, of the pairs $(\a,\b)$ where $\a$ is an Einstein metric and $\b$ is a Killing form of $\a$ satisfying all the restrictions in Theorem \ref{uenuangandgng}. It is clearly that the related deformations in Lemma \ref{wyeyyabndgagd} generate a transformation group $G$ acting on $\Lambda$. In particular, all the non Ricci-flat Einstein metrics and their Killing forms satisfying such conditions can be obtained by Ricci-flat Einstein metrics  and their Killing forms satisfying the same conditions, and vice versa.

On the other hand, the subset of $\Lambda$ with $\a$ being of constant sectional curvature is closed and transitive under the action of $G$. We don't know whether exist data $(\a,\b)$ in which $\a$ is of non-constant sectional curvature in $\Lambda$.
Perhaps there is only one orbit.

Our purpose is to construct new Einstein metrics using some known ones. Combining with Lemma \ref{bbchange}, it is easy to see that we can only consider those metrics which are Ricci flat if we want to do it locally. Meanwhile, we can, without loss of generality, always take $a=1$. In this case, the condition (\ref{owemigneigaindnab}) is simplified as
\begin{eqnarray}\label{doandnnanngnnga}
	\too=-\a^2.
\end{eqnarray}
However, sometimes maybe it is more natural to take the deformation beginning with a non Ricci-flat metric instead of a Ricci-flat metric. See Example \ref{doaieningangga} in Section \ref{adiannnndnnggbabbggg} for instance.

\begin{remark}\label{dinnaindbagfffafdg}
	Notice that $\{s_{ij}\}$ is an antisymmetric matrix, (\ref{doandnnanngnnga}) implies that the rank of $\{a_{ij}\}$ must be even. Combining with the about discussion, we know that the crucial condition (\ref{owemigneigaindn}) can only happen on even dimensional manifolds. Moreover, (\ref{doandnnanngnnga}) implies that $\{s_{ij}\}$ owns maximal rank.
\end{remark}

\section{General solutions of equations $X=T=0$ and a classification result}\label{dimiunuang}
\begin{lemma}\label{doadvvvvadgmanbg}
	Assume $n=2m$, $m\geq2$ is an integer. The solutions of $X=0$ and $T=0$ with a restriction $1+2b^2\rho'\neq0$ can be determined by
	\begin{eqnarray}
		\Delta(b^2)=\f{(a^2-\mu b^2)(1+2b^2\rho')^2}{a^2(1+Eb^2e^{2\rho})}\label{duaudbabdbgadm}
	\end{eqnarray}
	and
	\begin{eqnarray}
		\ln(b^2)=\int f(\varrho)\,\ud\varrho~(\mu=0)\label{iadnnanadgnada}
	\end{eqnarray}
	or
	\begin{eqnarray}
		\ln\left|\frac{a-\sqrt{a^2-\mu b^2}}{a+\sqrt{a^2-\mu b^2}}\right|=\int f(\varrho)\,\ud\varrho~(\mu\neq0),\label{iadnnanadgnadb}
	\end{eqnarray}
	where $E$ is a constant, $\varrho:=b^2e^{2\rho}$. The function $f$ is difined by 
	\begin{eqnarray*}
		f(\varrho)=\pm\frac{1}{\sqrt{F\varrho^{2-m}+\varrho^2-\frac{2m-1}{2m+2}\cdot\frac{\bar\mu}{a^2}\varrho^3}}
	\end{eqnarray*}
	when $E=0$, or
	\begin{eqnarray*}
		f(\varrho)=\pm\frac{\varrho^\frac{m-2}{2}}{\sqrt{F\iota^3-E^{-m-1}\iota^2\left\{\frac{\bar\mu}{a^2}\iota^{2m}
				+\sum\limits_{k=0}^{m-1}\frac{(-1)^{m-k}}{2k-1}\binom{m}{k}\left[2(m-k)E+(2m-1)\frac{\bar\mu}{a^2}\right]
				\iota^{2k}\right\}}}
	\end{eqnarray*}
	when $E\neq0$, in which $F$ is a constant, $\iota:=\sqrt{1+E\varrho}$.
\end{lemma}
\begin{proof}
	Solving the equation $T=0$ yields (\ref{duaudbabdbgadm}). As a result, the equation $X=0$ reads
	\begin{eqnarray}\label{yetangebg}
		&&na^2+2\mu b^2(1+2b^2\rho')-\left\{(n-3)(a^2-\mu b^2)+\f{3(a^2-\mu b^2)}{1+Eb^2e^{2\rho}}\right\}(1+2b^2\rho')^2\nonumber\\
		&&-4(a^2-\mu b^2)b^2(1+2b^2\rho')'-(n-1)\bar\mu b^2e^{2\rho}=0,
	\end{eqnarray}
	or equivalently,
	\begin{eqnarray}\label{uabebbgagad}
		&&4(a^2-\mu b^2)b^4\varrho\varrho''+(a^2-\mu b^2)\left(n-7+\f{3}{1+E\varrho}\right)b^4\varrho'^2\nonumber\\
		&&+2(2a^2-3\mu b^2)b^2\varrho\varrho'-na^2\varrho^2+(n-1)\bar\mu\varrho^3=0.
	\end{eqnarray}
	where $\varrho(b^2):=b^2e^{2\rho}$.
	
	Since $1+2b^2\rho'\neq0$, $\varrho$ is not a constant. Take $u=\ln(b^2)$, $v=\varrho>0$ when $\mu=0$ and $u=\ln\left|\frac{a-\sqrt{a^2-\mu b^2}}{a+\sqrt{a^2-\mu b^2}}\right|$, $v=\varrho>0$ when $\mu\neq0$, then
	\begin{eqnarray*}
		\varrho'(b^2)=\frac{\ud u}{\ud b^2}\cdot\left(\f{\ud u}{\ud v}\right)^{-1},\quad
		\varrho''(b^2)=\frac{\ud^2 u}{\ud(b^2)^2}\cdot\left(\f{\ud u}{\ud v}\right)^{-1}-\left(\frac{\ud u}{\ud b^2}\right)^2\cdot\f{\ud^2 u}{\ud v^2}\cdot\left(\f{\ud u}{\ud v}\right)^{-3}.
	\end{eqnarray*}
	As a result, (\ref{uabebbgagad}) turns to be a linear equation about $\frac{1}{f^2}$, where $f(v):=\frac{\ud u}{\ud v}$, as below
	\begin{eqnarray*}
		2a^2(1+Ev)v\frac{\ud\frac{1}{f^2}}{\ud v}+a^2[(n-7)(1+Ev)+3]\frac{1}{f^2}+(1+Ev)v^2[(n-1)\bar\mu v-na^2]=0.
	\end{eqnarray*}
	Its solutions are listed in Lemma \ref{doadvvvvadgmanbg}.\qed
\end{proof}
\begin{remark}\label{dinnabbbbgmamgbg}
	When $E\neq0$, $f(\varrho)$ can also be expressed as
	\begin{eqnarray*}
		f(\varrho)=\pm\frac{\varrho^\frac{m-2}{2}}{\sqrt{F\iota^3+\iota^2 \varrho^m-(2m-1)E^{-m-1}(\frac{\bar\mu}{a^2}+E)\iota^3\int\iota^{-2}(\iota^2-1)^m\,\ud\iota}}.
	\end{eqnarray*}
\end{remark}

For the Euclidean metric, there exist $n-1$ $1$-forms $\sigma_1$, $\sigma_2$, $\cdots$, $\sigma_{n-1}$ such that it can be expressed locally as a warped product
\begin{eqnarray*}
	\a^2=\ud r^2+r^2(\sigma_1^2+\sigma_2^2+\cdots+\sigma_{n-1}^2),
\end{eqnarray*}
and meanwhile the $1$-form $\b$ in (\ref{inbadbgamgd}) is given by $\b=r^2\sigma_1$. According to Theorem \ref{tebhagnagad}, the following metric
\begin{eqnarray*}
	\bar\a^2=\f{1}{(C+Dr^2)^2}\ud r^2+\f{r^2}{(C+Dr^2)^2}(\sigma_1^2+\sigma_2^2+\cdots+\sigma_{n-1}^2)
\end{eqnarray*}
has constant sectional curvature $4CD$ and the $1$-form $\bb=\f{ar^2}{(C+Dr^2)^2}\sigma_1$ satisfies our additional condition (\ref{owemigneigaindn}). The above metric can be normalized as
\begin{eqnarray*}
	\bar\a^2=\frac{1}{1-\bar\mu\bar r^2}\ud\bar r^2+\bar r^2(\sigma_1^2+\sigma_2^2+\cdots+\sigma_{n-1}^2)
\end{eqnarray*}
after reparameterization by setting $\bar r=\f{r}{C+Dr^2}$. In this case $\bb=a\bar r^2\sigma_1$. In particular, (\ref{inbadbgamgddanda}) can be expressed as a warped product locally. As a result, we obtain the following classification.

\begin{theorem}\label{ainnandnngnadg}
	Assume $\a$ has constant sectional curvature $\mu$ and is expressed locally as
	\begin{eqnarray*}
		\a^2=\frac{1}{1-\mu r^2}\ud r^2+r^2(\sigma_1^2+\sigma_2^2+\cdots+\sigma_{n-1}^2),
	\end{eqnarray*}
	and $\b=ar^2\sigma_1$ is a Killing form of $\a$ satisfying (\ref{owemigneigaindn}). Then after $\b$-deformation, $\ba$ has constant Ricci curvature $\bar\mu$ if and only if, after necessary renormalization, $\bar\a$ can be expressed as
	\begin{eqnarray}\label{soaidnagnnaongnga}
		\bar\a^2&=&a^4r^4f^2(a^2r^2)\ud r^2+\f{1}{1+Ea^2r^2}\cdot\f{1}{a^4r^2f^{2}(a^2r^2)}\sigma_1^2+r^2(\sigma_2^2+\cdots+\sigma_{n-1}^2),
	\end{eqnarray}
	where $f$ is defined in Lemma \ref{doadvvvvadgmanbg}.
\end{theorem}
\begin{proof}
	Notice that $b^2=a^2r^2$. By Theorem \ref{uenuangandgng}, after $\b$-deformation,
	\begin{eqnarray*}
		\ba^2&=&e^{2\rho(b^2)}(\alpha^2-\kappa(b^2)\b^2)\nonumber\\
		&=&\frac{e^{2\rho(a^2r^2)}}{1-\mu r^2}\ud r^2+e^{2\rho(a^2r^2)}\Delta(a^2r^2)r^2\sigma_1^2
		+e^{2\rho(a^2r^2)}r^2(\sigma_2^2+\cdots+\sigma_{n-1}^2)
	\end{eqnarray*}
	has constant Ricci curvature $\bar\mu$ if and only if $\Delta$ and $\rho$ satisfy Eqs. (\ref{yeangenga}).
	
	In general, $\rho$ is not an elementary function. However, $\ba$ can be expressed explicitly by some suitable variable substitution.
	
	Take $\bar r:=\sqrt{\varrho(a^2r^2)}/a$, then by (\ref{iadnnanadgnada})-(\ref{iadnnanadgnadb}) we have
	\begin{eqnarray}
		\ud r=ar\sqrt{a^2-\mu a^2r^2}\cdot \bar rf(a^2\bar r^2)\,\ud \bar r.\label{donnngggabbdbg}
	\end{eqnarray}
	On the other hand, (\ref{duaudbabdbgadm}) and (\ref{donnngggabbdbg}) yields
	\begin{eqnarray*}
		\Delta(a^2r^2)
		&=&\frac{1-\mu r^2}{1+Ea^2\bar r^2}\cdot\left[\f{r^2\varrho'(a^2r^2)}{\bar r^2}\right]^2\\
		&=&\frac{1-\mu r^2}{1+Ea^2\bar r^2}\cdot\left[\f{r^2\ud(\bar r^2)}{\bar r^2\ud(r^2)}\right]^2\\
		&=&\f{1}{1+Ea^2\bar r^2}\cdot\f{1}{a^4\bar r^{4}f^{2}(a^2\bar r^2)}.\label{donnngggabbdbg2}
	\end{eqnarray*}
	Hence,
	\begin{eqnarray*}
		\ba^2&=&\frac{e^{2\rho(a^2r^2)}}{1-\mu r^2}\ud r^2+e^{2\rho(a^2r^2)}\Delta(a^2r^2)r^2\sigma_1^2
		+e^{2\rho(a^2r^2)}r^2(\sigma_2^2+\cdots+\sigma_{n-1}^2)\\
		&=&a^4\bar r^4f^2(a^2\bar r^2)\ud \bar r^2+\f{1}{1+Ea^2\bar r^2}\cdot\f{1}{a^4\bar r^2f^{2}(a^2\bar r^2)}\sigma_1^2+\bar r^2(\sigma_2^2+\cdots+\sigma_{n-1}^2).
	\end{eqnarray*}\qed
\end{proof}
\begin{remark}
	After scaling $\dot{\ba}=\frac{1}{k}\ba$, $\dot{\ba}$ can be expressed as the form (\ref{soaidnagnnaongnga}) by reparameterization $r\rightarrow kr$, with the corresponding quantities determined by $\dot{\bar\mu}=k^2\bar\mu$, $\dot E=k^2E$ and $\dot F=k^{-n}F$. Hence, Theorem \ref{ainnandnngnadg} actually provides a two-parameters Einstein metric for any given dimension $n$ and Ricci constant $\bar\mu$.
\end{remark}

\section{$4$-dimensional Einstein metrics}\label{sinainhhdngnnag}
This is a special dimension not only in mathematics but also in physics. Here we just list all the Riemann-Einstein metrics that we can obtain by our way. But it is clear that our method is suitable for Lorentz metrics too.

Take the standard Cartan-Maurer forms on $\mathbb S^3$
\begin{eqnarray*}
	&\sigma_1=\f{1}{r^2}(x^2\,\ud x^1-x^1\,\ud x^2+x^4\,\ud x^3-x^3\,\ud x^4),&\\
	&\sigma_2=\f{1}{r^2}(x^3\,\ud x^1-x^4\,\ud x^2-x^1\,\ud x^3+x^2\,\ud x^4),&\\
	&\sigma_3=\f{1}{r^2}(x^4\,\ud x^1+x^3\,\ud x^2-x^2\,\ud x^3-x^1\,\ud x^4),&
\end{eqnarray*}
in which $r:=\sqrt{(x^1)^2+(x^2)^2+(x^3)^2+(x^4)^2}$. Then the standard Euclidean metric on $\mathbb R^4$ is given by
\begin{eqnarray*}
	g_{\mathbb R^4}=\ud r^2+r^2(\sigma_1^2+\sigma_2^2+\sigma_3^2).
\end{eqnarray*}
Obviously, all the $1$-forms $r^2\sigma_i~(i=1,2,3)$ satisfy (\ref{owemigneigaindn}) with $\mu=0$ and $a=1$. Hence, all the Einstein metrics constructed by $\b$-deformations combining with data $\a=\sqrt{g_{\mathbb R^4}}$ and $\b=r^2\sigma_1$ can be determined completely due to Theorem \ref{ainnandnngnadg}.

\begin{theorem}\label{iud7dgdbg}
	Assume $\sigma_1$, $\sigma_2$, $\sigma_3$ be the Cartan-Maurer forms on $\mathbb S^3$. Then any Einstein metric which can be expressed locally as the form
	\begin{eqnarray*}
		\psi(r^2)\ud r^2+\phi(r^2)\sigma_1^2+\varphi(r^2)(\sigma_2^2+\sigma_3^2)
	\end{eqnarray*}
	must be, after normalization, one of
	\begin{eqnarray}\label{aodmaindbgbbag}
		\ \ \ \bar\a^2=\left(1+Fr^{-4}-\frac{\bar\mu}{2}r^2\right)^{-1}\ud r^2+\left(1+Fr^{-4}-\frac{\bar\mu}{2}r^2\right)r^2\sigma_1^2+r^2(\sigma_2^2+\sigma_{3}^2)
	\end{eqnarray}
	or
	\begin{eqnarray}\label{aodmaindbgbbagb}
		\bar\a^2=\Omega\ud r^2+(1+Er^2)^{-1}\Omega^{-1}r^2\sigma_1^2+r^2(\sigma_2^2+\sigma_{3}^2),
	\end{eqnarray}
	where 
	\begin{eqnarray*}
		\Omega:=\frac{r^4}{(1+Er^2)\left\{F\sqrt{1+Er^2}-E^{-3}\left[\bar\mu(1+Er^2)^2-2(2E+3\bar\mu)(1+Er^2)-(4E+3\bar\mu)\right]\right\}}.
	\end{eqnarray*}
\end{theorem}

(\ref{aodmaindbgbbag}) are Eguchi-Hanson's metrics (3.24) in \cite{EH2}. It is known that the Ricci-flat Eguchi-Hanson's metric($\bar\mu=0$) is
K\"{a}hler on the cotangent bundle of $\mathbb S^2$\cite{Dancer}.

Moreover, we have
\begin{itemize}
	\item When $F=0$ and $\bar\mu=4$, (\ref{aodmaindbgbbag}) is the Fubini-Study metric on $\mathbb{CP}^2$, which can be expressed as\cite{Hitchin}
	\begin{eqnarray*}
		g_{\mathrm{FS}}=\f{1}{2}\left\{\f{1}{(1+t^2)^2}(\ud t^2+t^2\sigma_1^2)+\f{t^2}{1+t^2}(\sigma_2^2+\sigma_3^2)\right\}
	\end{eqnarray*}
	by the reparameterization $r=\frac{1}{\sqrt{2}}\cdot\frac{t}{\sqrt{1+t^2}}$.
	\item When $F=0$ and $\bar\mu=-4$, (\ref{aodmaindbgbbag}) is the Bergmann metric on $\mathbb{B}^4$, which can be expressed as\cite{Hitchin}
	\begin{eqnarray*}
		g_{\mathrm{B}}=\f{1}{2}\left\{\f{\ud t^2}{(1-t^2)^2}+\f{t^2}{(1-t^2)^2}\sigma_1^3+\f{t^2}{1-t^2}(\sigma_2^2+\sigma_3^2)\right\}
	\end{eqnarray*}
	by the reparameterization $r=\frac{1}{\sqrt{2}}\cdot\frac{t}{\sqrt{1-t^2}}$.
	\item When $E=\frac{4}{m^2}$, $F=-\frac{m^4}{2}$ and $\bar\mu=0$, (\ref{aodmaindbgbbagb}) is the Hawking Taub-NUT metric on $\mathbb{R}^4$\cite{Hawking,LeBrun}. See Theorem \ref{aondbabdgmmadbgg} for related discussion. By the reparameterization $r=t\sqrt{t^2+m}$, it will turn to be a more familiar face
	\begin{eqnarray*}
		g_{\mathrm{TN}}=(t^2+m)\ud t^2+\f{m^2t^2}{t^2+m}\sigma_1^2+t^2(t^2+m)(\sigma_2^2+\sigma_3^2).
	\end{eqnarray*}
	\item When $E=4(m+1)$ and $F=\frac{m}{2(m+1)^3}$ and $\bar\mu=-4$, (\ref{aodmaindbgbbagb}) belong to Pedersen's metrics\cite{Pedersen2,Ymatsumoto}, the original expression of which is given by
	\begin{eqnarray*}
		g_{\mathrm{P}}=\f{1}{(1-t^2)^2}\left\{\f{1+mt^2}{1+mt^4}\ud t^2+\f{t^2(1+mt^4)}{1+mt^2}\sigma_1^2+t^2(1+mt^2)(\sigma_2^2+\sigma_3^2)\right\}.
	\end{eqnarray*}
	It is known that such a metric is complete on $\mathbb{B}^4$ when $m\geq-1$. One can obtain our expression by the reparameterization $t=\sqrt{\frac{2 r^2}{1-mr^4}}$.
\end{itemize}

\begin{remark}
	$n=8$ is also a special dimension. We have a global frame on $\mathbb{S}^7$
	\begin{eqnarray*}
		&\sigma_1:=\f{1}{r^2}(-x^2\,\ud x^1+x^1\,\ud x^2-x^4\,\ud x^3+x^3\,\ud x^4-x^6\,\ud x^5+x^5\,\ud x^6-x^8\,\ud x^7+x^7\,\ud x^8),&\\
		&\sigma_2:=\f{1}{r^2}(-x^3\,\ud x^1+x^4\,\ud x^2+x^1\,\ud x^3-x^2\,\ud x^4-x^7\,\ud x^5-x^8\,\ud x^6+x^5\,\ud x^7+x^6\,\ud x^8),&\\
		&\sigma_3:=\f{1}{r^2}(-x^4\,\ud x^1-x^3\,\ud x^2+x^2\,\ud x^3+x^1\,\ud x^4-x^8\,\ud x^5+x^7\,\ud x^6-x^6\,\ud x^7+x^5\,\ud x^8),&\\
		&\sigma_4:=\f{1}{r^2}(-x^5\,\ud x^1+x^6\,\ud x^2+x^7\,\ud x^3+x^8\,\ud x^4+x^1\,\ud x^5-x^2\,\ud x^6-x^3\,\ud x^7-x^4\,\ud x^8),&\\
		&\sigma_5:=\f{1}{r^2}(-x^6\,\ud x^1-x^5\,\ud x^2+x^8\,\ud x^3-x^7\,\ud x^4+x^2\,\ud x^5+x^1\,\ud x^6+x^4\,\ud x^7-x^3\,\ud x^8),&\\
		&\sigma_6:=\f{1}{r^2}(-x^7\,\ud x^1-x^8\,\ud x^2-x^5\,\ud x^3+x^6\,\ud x^4+x^3\,\ud x^5-x^4\,\ud x^6+x^1\,\ud x^7+x^2\,\ud x^8),&\\
		&\sigma_7:=\f{1}{r^2}(-x^8\,\ud x^1+x^7\,\ud x^2-x^6\,\ud x^3-x^5\,\ud x^4+x^4\,\ud x^5+x^3\,\ud x^6-x^2\,\ud x^7+x^1\,\ud x^8),&
	\end{eqnarray*}
	which is similar to the Cartan-Maurer forms on $\mathbb S^3$, and the $8$-dimensional standard Euclidian metric is
	\begin{eqnarray*}
		g_{\mathbb R^8}=\ud r^2+r^2(\sigma_1^2+\sigma_2^2+\sigma_3^2+\sigma_4^2+\sigma_5^2+\sigma_6^2+\sigma_7^2).
	\end{eqnarray*}
	Obviously, all the $1$-forms $r^2\sigma_i~(i=1,\cdots,7)$ satisfy (\ref{owemigneigaindn}) with $\mu=0$ and $a=1$. So we can obtain many Einstein metrics by Theorem \ref{ainnandnngnadg}. But for the other dimensions, it seems that there is not such a global frame on the sphere in general.
\end{remark}

\section{Some other typical examples}\label{adiannnndnnggbabbggg}
If a Riemannian metric can be expressed as $\psi(r^2)\ud r^2+\phi(r^2)\sigma_1^2+\varphi(r^2)(\sigma_2^2+\cdots+\sigma_{n-1}^2)$, then there are infinity many ways to express it as a similar form by taking any suitable variable substitution $r=r(\bar r)$. However, it may not necessarily exist a particular coordinate $\bar r$ such that all the functions $\psi$, $\phi$ and $\varphi$ are explicit.

In this sense, it is lucky, due to Theorem \ref{ainnandnngnadg}, that all the Einstein metrics under our assumptions can always be expressed as a warped product form with the corresponding coefficients being elementary functions.

In this section, we would like to express these metrics in a more concise form, namely the form (\ref{Drho}). Unfortunately, only a few part of deformation factors $\Delta$ and $\rho$ can be expressed explicitly. A summarization is listed below.

\begin{table}[htbp]
	\centering
	\label{table1}
	\begin{tabular}{|c|c|c|c|}
		\hline
		\scriptsize{Parameters}&\scriptsize{Dimension}&\scriptsize{Completeness}&\scriptsize{Corresponding metrics}\\
		\hline
		\scriptsize{$E=0$, $F=0$, $\bar\mu=0$}&\scriptsize{$n=n$}&\scriptsize{whole}&\scriptsize{Euclidean metrics}\\ 
		\hline
		\scriptsize{$E=0$, $F\neq0$, $\bar\mu=0$}&\scriptsize{$n=n$}&\scriptsize{whole}&\scriptsize{Example \ref{oanidnnanndgbbg}: general Eguchi-Hanson metric (\ref{generalEguchi-Hanson})}\\
		\hline
		\scriptsize{$E=0$, $F=0$, $\bar\mu\neq0$}&\scriptsize{$n=n$}&\scriptsize{whole}&\scriptsize{Example \ref{doaieningangga}: Fubini-Study or Bergmann metrics (\ref{danidnnabbgbbgag})}\\
		\hline
		\scriptsize{$E=0$, $F\neq0$, $\bar\mu\neq0$}&\scriptsize{$n=4$}&\scriptsize{whole}&\scriptsize{Example \ref{dubaudvajdvgjkhbag}: non Ricci-flat Eguchi-Hanson's metrics}\\
		\hline
		\scriptsize{$E\neq0$, $\bar\mu=0$}&\scriptsize{$n=4$}&\scriptsize{whole}&\scriptsize{Example \ref{dnianndnignangnagng}: 
			Generalized Hawking Taub-NUT metrics}\\
		\hline
		\scriptsize{$E\neq0$, $F=0$, $\bar\mu\neq0$}&\scriptsize{$n=n$}&\scriptsize{
			partial}&\scriptsize{Sphere or Hyperbolic metrics}\\
		\hline
		\scriptsize{$E\neq0$, $\bar\mu\neq0$}&\scriptsize{$n=4$}&\scriptsize{whole}&\scriptsize{Example \ref{dubavffaccddhgga}: include Pedersen metrics}\\
		\hline
	\end{tabular}
	\caption{Some Einstein metrics expressed explicitly in the form (\ref{Drho})}
\end{table}

Obviously that $E=0$ is simpler than $E\neq0$. When $E=F=0$ and $\bar\mu=0$, then $\ba$ is an Euclidean metric according to Theorem \ref{ainnandnngnadg}. Beyond this, three other cases are discussed below.

\begin{example}[$E=0$, $F\neq0$, $\bar\mu=0$]\label{oanidnnanndgbbg}
	Take $\mu=0$, $\bar\mu=0$ and $E=0$. By substitution $\rho=-\ln(b^2)+\frac{2}{n}\ln g(b^2)$, Eq. (\ref{yetangebg}) turns to be a linear equation about $g$,
	\begin{eqnarray*}
		2b^2g''-(n-2)g'=0.
	\end{eqnarray*}
	So
	\begin{eqnarray}\label{dniaggdbngnabvggg}
		\Delta(b^2)=\left(\f{C-Db^{n}}{C+Db^{n}}\right)^2,\quad\rho(b^2)=-\ln b^2+\frac{2}{n}\ln(C+Db^{n}),
	\end{eqnarray}
	in which $C$ and $D$ are constants. By (\ref{Drho}), the following metrics
	\begin{eqnarray}
		\ba=\f{\sqrt{(C+Db^n)^2(b^2\a^2-\b^2)+(C-Db^n)^2\b^2}}{b^3(C+Db^n)^{1-\frac{2}{n}}}\label{dnaggdbbgklfbabng}
	\end{eqnarray}
	are Ricci-flat when $\a$ is Ricci flat and $\b$ satisfies (\ref{owemigneigaindna})-(\ref{tuabgengand}).
	
	In particular, by taking $C=1$, $D=\mu$ and using data (\ref{inbadbgamgd}), we obtain the following Ricci-flat metrics
	\begin{eqnarray}\label{generalEguchi-Hanson}
		\ba=\f{\sqrt{(1+\mu|x|^n)^2|x|^2|y|^2-4\mu|x|^n\langle Jx,y\rangle^2}}{|x|^3(1+\mu|x|^n)^{1-\frac{2}{n}}}.
	\end{eqnarray}
	
	The congruent relationship between (\ref{dniaggdbngnabvggg}) and (\ref{iadnnanadgnada}) is determined by
	\begin{eqnarray*}
		f(\varrho)=\pm\frac{1}{\sqrt{\varrho^2-4CD\varrho^{2-m}}},
	\end{eqnarray*}
	namely $F=-4CD$. As a result, (\ref{generalEguchi-Hanson}) can be expressed locally as
	\begin{eqnarray}\label{E0bmu0a}
		\bar\a^2=\left(1+Ft^{-n}\right)^{-1}\ud t^2+\left(1+Ft^{-n}\right){t^2}\sigma_1^2+t^2(\sigma_2^2+\cdots+\sigma_{n-1}^2).
	\end{eqnarray}
	
	It is known that (\ref{E0bmu0a}) is the Ricci-flat Eguchi-Hanson metric when $n=4$\cite{EH1,EH2}~(See also \cite{Brendle} and \cite{Hitchin} for related discussions). Hence, (\ref{generalEguchi-Hanson}) can be regarded as the general Eguchi-Hanson metrics on higher dimensional spaces. In 2022, C. Corral et al. constructed these metrics independently from a physical point of view\cite{Corral}.
\end{example}
\begin{remark}
	According to Lemma \ref{bbchange}, the metrics (\ref{dnaggdbbgklfbabng}) can also be obtained by using a non-Ricci flat Einstein metric and the corresponding $1$-form combining with the following deformation factors
	\begin{eqnarray*}
		\Delta(b^2)&=&\left\{\f{C|a+\sqrt{a^2-\mu b^2}|^{-\frac{n}{2}}-D|a-\sqrt{a^2-\mu b^2}|^{-\frac{n}{2}}}{C|a+\sqrt{a^2-\mu b^2}|^{-\frac{n}{2}}+D|a-\sqrt{a^2-\mu b^2}|^{-\frac{n}{2}}}\right\}^2,\\
		\rho(b^2)&=&\frac{2}{n}\ln\left(C|a+\sqrt{a^2-\mu b^2}|^{-\frac{n}{2}}+D|a-\sqrt{a^2-\mu b^2}|^{-\frac{n}{2}}\right),
	\end{eqnarray*}
	in which $C$ and $D$ are constants (not the same ones in (\ref{dniaggdbngnabvggg})).
\end{remark}

When $E=0$, $\mu=0$ and $\bar\mu\neq0$, Eq. (\ref{yetangebg}) will turn to to be
\begin{eqnarray*}
	2b^2g''+(n+2)g'+\frac{n(n-1)\bar\mu}{8a^2}g^{1+\frac{4}{n}}=0
\end{eqnarray*}
by substitution $\rho=\frac{2}{n}\ln g(b^2)$, or to be
\begin{eqnarray}\label{donandnanngga}
	2b^2hh''-(n+2)b^2(h')^2+(n+2)hh'=(n-1)\f{\bar\mu}{a^2}
\end{eqnarray}
by substitution $\rho=-\ln h(b^2)$.

Both of them are hard to be solved completely (some complicated solutions can be found in Example \ref{dubaudvajdvgjkhbag}). However, notice that the RHS of (\ref{donandnanngga}) is a constant, such equation can be easily turn to be a three-order ODE without $\bar\mu$ and $a^2$ as below
\begin{eqnarray}
	\f{(2b^2h'-h)''}{2b^2h'-h}=\frac{n+1}{2}\cdot\f{h''}{h}.\label{unaidnangbga}
\end{eqnarray}
Here we assume $2b^2h'-h$ is non-zero, since if not so, the solutions are trivial in our problem.

\begin{example}[$E=0$, $F=0$, $\bar\mu\neq0$]\label{doaieningangga}
	Obviously, $h''=0$ is the simplest case for (\ref{unaidnangbga}), by which we can obtain a series of explicit solutions of (\ref{yetangebg}) as below
	\begin{eqnarray*}
		\Delta(b^2)=\left(\f{C-Db^2}{C+Db^2}\right)^2,\quad\rho(b^2)=-\ln(C+Db^2),
	\end{eqnarray*}
	in which $C$ aan $D$ are constants. In this case, Eq. (\ref{donandnanngga}) indicates $\bar\mu=\f{4(n+2)}{n-1}CDa^2$.
	
	Hence, if $\a$ is Ricci flat, $\b$ satisfies (\ref{owemigneigaindna})-(\ref{tuabgengand}), then the following metrics
	\begin{eqnarray}\label{gabbdaffavccdcg}
		\ba=\f{\sqrt{(C+Db^2)^2(b^2\a^2-\b^2)+(C-Db^2)^2\b^2}}{b(C+Db^2)^2}.
	\end{eqnarray}
	have Ricci constant $\f{4(n+2)}{n-1}CDa^2$.
	
	In particular, by taking $C=\frac{1}{2}$, $D=\frac{\mu}{2}$ and using data (\ref{inbadbgamgd}), we obtain the following metrics
	\begin{eqnarray}\label{danidnnabbgbbgag}
		\ba=\f{2\sqrt{(1+\mu|x|^2)^2|y|^2-4\mu\langle Jx,y\rangle^2}}{(1+\mu|x|^2)^2},
	\end{eqnarray}
	which have Ricci constant $\bar\mu=\frac{n+2}{n-1}\mu$.
	
	For the special solutions (\ref{gabbdaffavccdcg}), the function $f(\varrho)$ in Lemma \ref{doadvvvvadgmanbg} is given by
	\begin{eqnarray*}
		f(\varrho)=\pm\frac{1}{\sqrt{\varrho^2-4CD\varrho^3}},
	\end{eqnarray*}
	which means $E=F=0$. Hence, it can be expressed locally as
	\begin{eqnarray}\label{doanidnanngnangsa}
		\bar\a^2=\left(1-\mu t^2\right)^{-1}\ud t^2+\left(1-\mu t^2\right)t^2\sigma_1^2+t^2(\sigma_2^2+\cdots+\sigma_{n-1}^2)
	\end{eqnarray}
	according to Theorem \ref{ainnandnngnadg}, by which we can see that (\ref{danidnnabbgbbgag}) is actually the Fubini-Study metric when $\mu>0$ and the Bergmann metric when $\mu<0$.
	
	Since the resulting metric is not Ricci-flat, sometimes it will be more convenient to begin with a non Ricci-flat Einstein metric. For instant, if you want to construct some Einstein metrics on a sphere, then considering a round metric instead of a flat metric is more natural.
	
	Due to Lemma \ref{bbchange} and Theorem \ref{tebhagnagad}, we can obtain the metrics (\ref{danidnnabbgbbgag}) beginning with a Riemannian metric with constant sectional curvature $\mu$ and its Killing $1$-form $\b$ satisfying (\ref{owemigneigaindn}) by $\b$-deformations determined by
	\begin{eqnarray*}
		\Delta(b^2)=\f{(C\sqrt{a^2-\mu b^2}+Da^2)^2}{a^2(C+D\sqrt{a^2-\mu b^2})^2},\quad\rho(b^2)=-\ln(C+D\sqrt{a^2-\mu b^2}),
	\end{eqnarray*}
	in which $C$ and $D$ are constants (not the same ones in (\ref{gabbdaffavccdcg})). In this case, $\bar\mu=\frac{n+2}{n-1}(C^2-D^2a^2)\mu$, and the resulting Einstein metrics read
	\begin{eqnarray}\label{donaidnndngadggg}
		\ \ \ \ba=\f{\sqrt{a^2(C+D\sqrt{a^2-\mu b^2})^2(b^2\a^2-\b^2)+(C\sqrt{a^2-\mu b^2}+Da^2)^2\b^2}}{|a|b(C+D\sqrt{a^2-\mu b^2})^2}.
	\end{eqnarray}
	
	It seems from the expression (\ref{gabbdaffavccdcg}) or (\ref{donaidnndngadggg}) that we obtain a two-parameters Einstein metrics. However, for a given constant $\bar\mu$ and a given dimension $n$, there is only one Einstein metric up to an isometry locally. This can be easily seen by the expression (\ref{doanidnanngnangsa}). Hence, the above particular solutions, namely the Fubini-Study metrics and Bergmann metrics, can be expressed in a simpler form as
	\begin{eqnarray}\label{FSB}
		\ba=\sqrt{\a^2-\mu\b^2}
	\end{eqnarray}
	by choosing $C=1$, $D=0$ and $a=1$ in (\ref{donaidnndngadggg}). One can obtain the expression (\ref{danidnnabbgbbgag}) again by using above formulae combining with data (\ref{inbadbgamgddanda}), and obtain a more common form by using data (\ref{inbadbgamgddandab}), namely
	\begin{eqnarray*}
		\ba=\f{\sqrt{(1+\mu|x|^2)|y|^2-\mu\langle x,y\rangle^2-\mu\langle Jx,y\rangle^2}}{1+\mu|x|^2}
		=\f{\sqrt{(1+\mu z\cdot\bar z)w\cdot\bar w-\mu (z\cdot\bar w)(\bar z\cdot w)}}{1+\mu z\cdot\bar z},
	\end{eqnarray*}
	where $z:=(z_1,z_2,\cdot,z_m)$, $w:=(w_1,w_2,\cdots,w_m)$, $z_k:=x_{2k-1}+ix_{2k}$, $w_k:=y_{2k-1}+iy_{2k}$ for $1\leq k\leq m$.
\end{example}
\begin{remark}
	Be attention that the resulting metric $\ba$ may not define on the initial manifold after $\b$-deformation. For example, consider a standard spherical metric $\a$ on $\mathbb{S}^{2n}$ and its $1$-form $\b$ satisfying (\ref{owemigneigaindn}). According to the above discussion, (\ref{FSB}) is the Fubini-Study metric defined on $\mathbb{CP}^{n}$ instead of $\mathbb{S}^{2n}$. 
	
	What happens under $\b$-deformation?
	
	In fact, $\b$-deformation maybe lead to some singularity. Due to (\ref{owemigneigaindn}),
	\begin{eqnarray*}
		t=-b^2(a^2-\mu b^2).
	\end{eqnarray*}
	Notice that $t=-s_ia^{ij}s_j$ can be regard as the opposite number of the inner product for the vector $s_i$, we can conclude that $a^2-\mu b^2$ must be non-negative. When $a^2-\mu b^2=0$, (\ref{owemigneigaindn}) implies $s_0=0$ at the corresponding points. For the round sphere, this will happen on the equator.
	
	For such points, the metric may be of  some singularity. Let's talk about a simple example, which is interesting but instructive. Taking a standard spherical metric $\a$ and its $1$-form $\b$ expressed locally as (\ref{inbadbgamgddanda}) with $\mu=1$ and $a=1$, then
	\begin{eqnarray*}
		a^2-\mu b^2=\f{(1-|x|^2)^2}{(1+|x|^2)^2}.
	\end{eqnarray*}
	As a result,
	\begin{eqnarray*}
		\ba=\f{2|y|}{1+|x|^2+\left|1-|x|^2\right|}
	\end{eqnarray*}
	is Ricci-flat by taking $C=D=1$ in (\ref{dinewbbabvavdb}). It is easy to see that on the southern hemisphere (namely the region $|x|<1$), $\ba=|y|$ is the flat metric, and on the northern hemisphere (namely the region $|x|>1$), $\ba=\f{|y|}{|x|^2}$ is the pullback metric of flat metric by inversion transformation with respect to the hypersurface $|x|=1$. On the equator $|x|=1$, the metric $\ba$ is continuous, but not smooth.
	
	We don't know how to deal with the singularity at this moment. It seems that, illuminated by Funibi-Study metrics, such singularity can possibly be repaired by taking some operations for the underlying manifold.
\end{remark}

\begin{example}[$E=0$, $F\neq0$, $\bar\mu\neq0$]\label{dubaudvajdvgjkhbag}
	For this case, we can only make clear the solutions for the simplest dimension. When $n=4$ and $a=1$, the function $f(\varrho)$ in Lemma \ref{doadvvvvadgmanbg} is given by
	\begin{eqnarray*}
		f(\varrho)=\pm\frac{1}{\sqrt{F+\varrho^2-\frac{\bar\mu}{2}\varrho^3}}.
	\end{eqnarray*}
	
	Denote the roots of $F+\varrho^2-\frac{\bar\mu}{2}\varrho^3=0$ by $\varrho_1$, $\varrho_2$, and $\varrho_3$. Due to Vieta theorem, the roots should satisfy a constraint conditions
	\begin{eqnarray*}
		\varrho_1\varrho_2+\varrho_2\varrho_3+\varrho_3\varrho_1=0.
	\end{eqnarray*}
	Moreover, $F\neq0$ implies $\varrho_1\varrho_2\varrho_3\neq0$.
	
	As a result, we can calculate the integral $\int f(\varrho)\,\ud\varrho$ and obtain deformation factors due to (\ref{iadnnanadgnada})-(\ref{iadnnanadgnadb}). When $\mu=0$, the solutions are described by Jacobi elliptic sine and cosine functions under the help of Maple program as below
	\begin{eqnarray*}
		\textstyle\Delta(b^2)&=&\f{\bar\mu(\varrho_2-\varrho_3)^2\mathrm{sn}^2(\xi\ln b^2,\eta)\mathrm{cn}^2(\xi\ln b^2,\eta)\left\{(\varrho_2-\varrho_3)\mathrm{sn}^2(\xi\ln b^2,\eta)
			+(\varrho_3-\varrho_1)\right\}}{2\left\{(\varrho_2-\varrho_3)\mathrm{sn}^2(\xi\ln b^2,\eta)
			+\varrho_3\right\}^2},\\
		\rho(b^2)&=&\f{1}{2}\ln\left\{(\varrho_2-\varrho_3)\mathrm{sn}^2(\xi\ln b^2,\eta)
		+\varrho_3\right\}-\f{1}{2}\ln b^2,
	\end{eqnarray*}
	in which
	\begin{eqnarray*}
		\xi=\sqrt{\frac{\bar\mu(\varrho_3-\varrho_1)}{8}},\quad\eta=\sqrt{\frac{\varrho_3-\varrho_2}{\varrho_3-\varrho_1}}.
	\end{eqnarray*}
	In this case, the corresponding metric has Ricci constant $\bar\mu=\frac{2}{\varrho_1+\varrho_2+\varrho_3}$.
	
	The above expressions asks $\varrho_1\neq\varrho_3$. Notice that triple root is not allowed, or else $\varrho_1=\varrho_2=\varrho_3=0$, which is in contradiction with the assumption that $F\neq0$. Hence, if $\varrho_1=\varrho_3$, in order to obtain the solution, we can exchange the value of $\varrho_2$ and $\varrho_3$. For this particular case, the corresponding metric, up to a re-scaling, is determined by
	\begin{eqnarray*}
		\Delta(b^2)&\f{\tan^2(\frac{1}{2}\ln b^2)\left[\tan^2(\frac{1}{2}\ln b^2)+1\right]^2}{\left[\tan^2(\frac{1}{2}\ln b^2)+\frac{1}{3}\right]^2},\quad
		\rho(b^2)=\frac{1}{2}\ln\left[\tan^2(\frac{1}{2}\ln b^2)+\frac{1}{3}\right]-\frac{1}{2}\ln b^2,
	\end{eqnarray*}
	and has Ricci constant $\bar\mu=-2$.
	
	According to the discussion in Section \ref{sinainhhdngnnag}, the above metrics are actual the non Ricci-flat Eguchi-Hanson's metrics.
\end{example}

When $E\neq0$, it is hard to determined the solutions for general dimensions. But we can provide all the solutions when $n=4$.

\begin{example}[$E\neq0$, $\bar\mu=0$]\label{dnianndnignangnagng}
	Under the help of Maple program, we obtain the following solutions for $4$-dimensional case,
	\begin{eqnarray*}
		\Delta(b^2)&=&\f{64DEb^4(b^4-C^2+4DE)^2}{\left\{(b^2+C)^4-8DE(3b^4+2Cb^2+C^2-2DE)\right\}^2},\\
		\rho(b^2)&=&\f{1}{2}\ln\f{(b^2+C)^4-8DE(3b^4+2Cb^2+C^2-2DE)}{16DE^2b^6},
	\end{eqnarray*}
	in which $C$ and $D$ are constants. $DE$ must be positive to make sure the positivity of $\Delta$.
	
	For these solutions, the function $f(\varrho)$ in Lemma \ref{doadvvvvadgmanbg} is given by
	\begin{eqnarray*}
		f(\varrho)=\pm\frac{1}{\sqrt{-\frac{4C}{E^2\sqrt{DE}}\iota^3+4E^{-2}(\iota^2+1)}}.
	\end{eqnarray*}
	So they provide all the solutions for this case when $n=4$. In particular, when $E=\frac{C^2}{4D}$, we have $\rho(b^2)=\frac{1}{2}\ln\frac{D(b^2+4C)}{C^4}$ and $\Delta(b^2)=\frac{16C^2}{(b^2+4C)^2}$, the corresponding metric is Hawking Taub-NUT metric(c.f. Theorem \ref{aondbabdgmmadbgg}). So such metrics can be regarded as the generalization of Hawking Taub-NUT metric in $4$-dimension.
\end{example}

The case that both $E$ and $\bar\mu$ are not zero is more complicated. It is easy to verify that, for the solutions (\ref{dinewbbabvavda}), the corresponding function $f$ is given by
\begin{eqnarray*} f(\varrho)=\pm\frac{1}{\varrho\sqrt{1-4CD\varrho}}=\pm\frac{\varrho^{\frac{m-2}{2}}}{\sqrt{(1-4CD\varrho)\varrho^m}}.
\end{eqnarray*}
Hence, $E=-4CD=-\frac{\bar\mu}{a^2}$ and $F=0$. By Theorem \ref{tebhagnagad}, the resulting metrics have constant sectional curvature. Due to Remark \ref{dinnabbbbgmamgbg} in Section \ref{dimiunuang}, one can see that $E=-\frac{\bar\mu}{a^2}$ is a special value. When $E\neq0$, these are the totally explicit solutions for general dimensions that we can obtain at present.

Finally, we will still focus on the $4$-dimensional case. In this case, the deformation factors can be expressed also by using Jacobi elliptic functions.

\begin{example}[$E\neq0$, $\bar\mu\neq0$]\label{dubavffaccddhgga}
	Assume $n=4$, $a=1$, then the function $f(\varrho)$ in Lemma \ref{doadvvvvadgmanbg} is given by
	\begin{eqnarray*}
		f(\varrho)=\pm\f{1}{\iota\sqrt{F\iota-E^{-3}\left\{\bar\mu\iota^4
				-2(2E+3\bar\mu)\iota^2-(4E+3\bar\mu)\right\}}},
	\end{eqnarray*}
	in which $\iota:=\sqrt{1+E\varrho}$. In this case, $\int f(\varrho)\,\ud\varrho$ is an incomplete elliptic integral of the first kind, and $\rho$ determined by (\ref{iadnnanadgnada})-(\ref{iadnnanadgnadb}) is a Jacobi elliptic function.
	
	Denote $k:=-4E/\bar\mu\neq0$ and $l:=-E^3F/\bar\mu$, then
	\begin{eqnarray*}
		\int f(\varrho)\,\ud\varrho=\pm\int\f{\ud\iota}{\sqrt{\frac{1}{k}\left\{\iota^4+(k-6)\iota^2+l\iota+k-3\right\}}}.
	\end{eqnarray*}
	
	Assume a real factorization as below
	\begin{eqnarray}
		\iota^4+(k-6)\iota^2+l\iota+k-3=(\iota^2+p\iota+q)(\iota^2-p\iota+r),\label{dobfbbhbagdgvg}
	\end{eqnarray}
	by which we have
	\begin{eqnarray*}
		q+r-p^2=k-6,\quad-p(q-r)=l,\quad qr=k-3.
	\end{eqnarray*}
	These equalities imply the real numbers $p$, $q$ and $r$ should satisfy the following constraint condition
	\begin{eqnarray}\label{dfaggdhhbalgkdkg}
		p^2+(q-1)(r-1)=4.
	\end{eqnarray}
	
	Denote the roots of $\iota^4+(k-6)\iota^2+l\iota+k-3=0$ by $\iota_1$, $\iota_2$, $\iota_3$ and $\iota_4$. Due to Vieta theorem, the roots should satisfy two constraint conditions
	\begin{eqnarray}
		&\iota_1+\iota_2+\iota_3+\iota_4=0,&\label{dinahhndbbvagvvga}\\
		&\iota_1\iota_2\iota_3\iota_4-\iota_1\iota_2-\iota_1\iota_3-\iota_1\iota_4
		-\iota_2\iota_3-\iota_2\iota_4-\iota_3\iota_4-3=0.&\label{dinahhndbbvagvvgb}
	\end{eqnarray}
	Moreover, $E\neq0$, namely $k\neq0$, implies $\iota_1\iota_2\iota_3\iota_4+3\neq0$.
	
	Combining with (\ref{dinahhndbbvagvvga}) and (\ref{dinahhndbbvagvvgb}), it is easy to see that triple root or quadruple root is not allowed. That is to say, there are only three possible cases for the roots: (i) exist two different double roots, (ii) exits only one double root and (iii) all the four roots are different from each other.
	
	\noindent{\em Case (i) exist two different double roots}
	
	By (\ref{dobfbbhbagdgvg}), we have
	\begin{eqnarray*}
		p^2-4(q-1)=4,\qquad p^2-4(r-1)=4
	\end{eqnarray*}
	or
	\begin{eqnarray*}
		p=0,\qquad q=r.
	\end{eqnarray*}
	
	Combining with (\ref{dfaggdhhbalgkdkg}), the former indicates $p=\pm2$ and $q=r=1$. The corresponding deformation factors are given by (\ref{dinewbbabvavda}). Due to Theorem \ref{tebhagnagad}, the metrics have constant sectional curvature.
	
	The latter indicates $q=r=-1$, which is just the former case, or $q=r=3$. When $q=r=3$, by solving the eqution (\ref{iadnnanadgnada}) with $\mu=0$, the deformation factors after re-scaling are given by
	\begin{eqnarray*}
		\Delta(b^2)=\f{\left[\tan^2(\frac{1}{2}\ln b^2)+1\right]^2}{3\left[\frac{1}{3}-\tan^2(\frac{1}{2}\ln b^2)\right]^2},
		\quad\rho(b^2)=\frac{1}{2}\ln\left[\frac{1}{3}-\tan^2(\frac{1}{2}\ln b^2)\right]-\frac{1}{2}\ln b^2,
	\end{eqnarray*}
	or
	\begin{eqnarray*}
		\Delta(b^2)=\f{\left[\tan^2(\frac{1}{2}\ln b^2)+1\right]^2}{3\left[\tan^2(\frac{1}{2}\ln b^2)-\frac{1}{3}\right]^2},\quad
		\rho(b^2)=\frac{1}{2}\ln\left[\tan^2(\frac{1}{2}\ln b^2)-\frac{1}{3}\right]-\frac{1}{2}\ln b^2,
	\end{eqnarray*}
	and the corresponding metric has Ricci constant $\bar\mu=1$ with $E=-3$ or $\bar\mu=-1$ with $E=3$ respectively.
	
	\noindent{\em Case (ii) exits only one double root}
	
	Without loss of generality, we can assume $\iota^2-p\iota+r=0$ has double root, which indicates $p^2-4r=0$. Combining it with (\ref{dfaggdhhbalgkdkg}) yields $(q+3)(r-1)=0$. Hence $q=-3$ or $r=1$.
	
	When $r=1$, then $p=\pm2$ due to (\ref{dfaggdhhbalgkdkg}), and the roots are given by $1$, $1$, $-1\pm\sqrt{1-q}$ or $-1$, $-1$, $1\pm\sqrt{1-q}$. Be attention that $q\neq1$, or else it can be reduced to Case (i). At the same time, $q\neq-3$, or else $k=0$. But $q>1$ is allowed, which means that exists a couple of conjugate complex roots. By solving the eqution (\ref{iadnnanadgnada})-(\ref{iadnnanadgnadb}), $\rho$ is given by
	\begin{eqnarray*}
		\frac{1}{2}\ln\left\{\f{-8(q+3)[b^4-2(q-1)b^2-4(q-1)]}{E[b^4+8b^2-4(q-1)]^2}\right\},
	\end{eqnarray*}
	\begin{eqnarray*}
		\frac{1}{2}\ln\left\{\f{4[b^4+\frac{q-1}{q+3}b^2-\frac{q-1}{(q+3)^2}]}{E[b^4-\frac{4}{q+3}b^2-\frac{q-1}{(q+3)^2}]^2}\right\}
	\end{eqnarray*}
	or
	\begin{eqnarray*}
		\frac{1}{2}\ln\left\{\frac{4(q+3)(b^4\pm\sqrt{q-1}b^2-1)}{\pm E\sqrt{q-1}(b^4\mp\frac{4}{\sqrt{q-1}}b^2-1)^2}\right\}
	\end{eqnarray*}
	when $q<-3$, $-3<q<1$ or $q>1$ respectively.
	
	These solutions can be uniformly expressed, after re-scaling, as
	\begin{eqnarray*}
		\Delta(b^2)=\f{(b^4-D)^2}{(b^4+D+2Cb^2)^2},\quad\rho(b^2)=\f{1}{2}\ln\left\{\f{C[(b^4+D)+2Cb^2]}{[C(b^4+D)+2 Db^2]^2}\right\},
	\end{eqnarray*}
	in which both $C$ and $D$ are non-zero constants. In this case, $E=8(C^2-D)$, and the corresponding metric have Ricci-constant $\bar\mu=8D$, which are Pedersen metrics in fact. Notice that $E\neq0$ asks $D\neq C^2$. If $D=C^2$, the above solutions are given in Example \ref{doaieningangga}.
	
	When $q=-3$, the roots are given by $\frac{p}{2}$, $\frac{p}{2}$, $-\frac{p}{2}\pm\sqrt{3+\frac{p^2}{4}}$. Be attention that $p^2\neq4$, or else $k=0$. In this case, the deformation factors are given by
	\begin{eqnarray*}
		\Delta(b^2)&=&\f{144(p^2+12)\cos^2\ln b^2}{\left\{\left[\sqrt{p^2+12}\sin(\ln b^2)+p\right]^2-36\right\}^2},\\
		\rho(b^2)&=&\f{1}{2}\ln\f{(p^2-4)\left\{\left[\sqrt{p^2+12}\sin(\ln b^2)+p\right]^2-36\right\}}{4E\left[\sqrt{p^2+12}\sin(\ln b^2)-2p\right]^2}-\f{1}{2}\ln b^2,
	\end{eqnarray*}
	and the corresponding metric has Ricci-constant $\bar\mu=\f{16E}{3(p^2-4)}$.
	
	\noindent{\em Case (iii) all the four roots are different from each other}
	
	Under the help of Maple program, we can obtain the deformation factors by solving the equation (\ref{iadnnanadgnada})-(\ref{iadnnanadgnadb}). When $\mu
	=0$, they are given by
	\begin{eqnarray*}
		\Delta(b^2)&=&\f{\sigma\mathrm{sn}^2(\xi\ln b^2,\eta)\mathrm{cn}^2(\xi\ln b^2,\eta)
			\left\{(\iota_1-\iota_2)(\iota_3-\iota_4)\mathrm{sn}^2(\xi\ln b^2,\eta)-(\iota_1-\iota_3)(\iota_2-\iota_4)\right\}}{\left\{\left[\f{(\iota_1-\iota_2)\iota_3\mathrm{sn}^2(\xi\ln b^2,\eta)-(\iota_1-\iota_3)\iota_2}{(\iota_1-\iota_2)\mathrm{sn}^2(\xi\ln b^2,\eta)-(\iota_1-\iota_3)}\right]^2-1\right\}^2\left\{(\iota_1-\iota_2)\mathrm{sn}^2(\xi\ln b^2,\eta)-(\iota_1-\iota_3)\right\}^4},\\
		\rho(b^2)&=&\f{1}{2}\ln\f{1}{E}\left\{\left[\f{(\iota_1-\iota_2)\iota_3\mathrm{sn}^2(\xi\ln b^2,\eta)-(\iota_1-\iota_3)\iota_2}{(\iota_1-\iota_2)\mathrm{sn}^2(\xi\ln b^2,\eta)-(\iota_1-\iota_3)}\right]^2-1\right\}-\f{1}{2}\ln b^2,
	\end{eqnarray*}
	in which
	\begin{eqnarray*}
		\xi=\sqrt{-\f{(\iota_1-\iota_3)(\iota_2-\iota_4)}{4(\iota_1\iota_2\iota_3\iota_4+3)}},\ \ \eta=\sqrt{\frac{(\iota_1-\iota_2)(\iota_3-\iota_4)}{(\iota_1-\iota_3)(\iota_2-\iota_4)}},\ \ 
		\sigma=\frac{4(\iota_1-\iota_2)^2(\iota_1-\iota_3)^2(\iota_2-\iota_3)^2}{\iota_1\iota_2\iota_3\iota_4+3}.
	\end{eqnarray*}
	The corresponding metric has Ricci constant $\bar\mu=-\frac{4E}{\iota_1\iota_2\iota_3\iota_4+3}$.
\end{example}
\begin{remark}
	In fact, the solutions for Case (ii) can also be expressed as the expression in Case (iii), by putting the roots in the right place to make sure $(\iota_1-\iota_3)(\iota_2-\iota_4)\neq0$ and $(\iota_1-\iota_2)(\iota_1-\iota_3)(\iota_2-\iota_3)\neq0$. But the solutions in Case (i) can't be expressed as such form.
\end{remark}

\noindent{\bf Acknowledgement} 

The author would like to thank Prof. Z. Shen for his constant encouragement and helpful comments.

\setcounter{section}{10}
\setcounter{equation}{0}
\section*{Appendix Deformation formulae for Riemann tensor}
The behavior of the Riemannian curvature and Ricci curvature under $\b$-deformations are stated in the applendix. They will be useful for the further study. The whole derivation, which is done with the help of Maple program, is elementary but very tedious, so it is omitted here. All the results are checked carefully to make sure that any minor mistake is not allowed.

\begin{proposition}\label{RIkunderdeformaion}
	After $\b$-deformation,
	\begin{eqnarray*}
		\bRIk&=&\RIk+\big\{c_{101}\a^2 r^2+c_{102}\a^2(p-2q-t)+c_{103}\b^2r^2+c_{104}\b^2(p-2q-t)+c_{105}\b\ro r\\
		&&+c_{106}\b\so r+c_{107}\b(\qo+\to)+c_{108}\roo r+c_{109}(\ro+\so)^2+c_{1010}(\roho+\soho)\big\}\dIk\\
		&&+\big\{c_{201}r^2+c_{202}(p-2q-t)\big\}\yI\yk+\big\{c_{301}\ro r+c_{302}\so r+c_{303}(\qo+\to)+c_{304}\b r^2\\
		&&+c_{305}\b(p-2q-t)\big\}\yI\bk+c_4 r\yI\rko+\big\{c_{501}(\ro+\so)+c_{502}\b r\big\}\yI\rk+\big\{c_{601}(\ro\\
		&&+\so)+c_{602}\b r\big\}\yI\sk+c_7\b\yI(\qk+\tk)+c_8\yI(\rohk+\sohk)+c_9\yI(\rkho+\skho)\\
		&&+\big\{c_{1001}(\ro+\so)r+c_{1002}(\po-\qqo)+c_{1003}\b p+c_{1004}\b q+c_{1005}\b t+c_{1006}r_{|0}\big\}\bI\yk\\
		&&+\big\{c_{1101}\qoo+c_{1102}\ro^2+c_{1103}\ro\so+c_{1104}\so^2+c_{1105}\roho+c_{1106}\soho+c_{1107}\b(\ro\\
		&&+\so)r+c_{1108}\b(\po-\qqo)+c_{1109}\b\qo+c_{1110}\b\to+c_{1111}\b r_{|0}+c_{1112}\a^2p+c_{1113}\a^2q\\
		&&+c_{1114}\a^2 t\big\}\bI\bk+c_{12}(\ro+\so)\bI\rko+(c_{1301}\ro+c_{1302}\so+c_{1303}\b r)\bI\sko+(c_{1401}\roo\\
		&&+c_{1402}\b\ro+c_{1403}\b\so+c_{1404}\a^2r+c_{1405}\b^2r)\bI\rk+(c_{1501}\roo+c_{1502}\b\ro+c_{1503}\b\so\\
		&&+c_{1504}\a^2r+c_{1505}\b^2r)\bI\sk+c_{16}\b\bI\qko+c_{17}\b\bI\qok+c_{18}\a^2\bI(\pk-\qqk)+c_{19}\b^2\bI\\
		&&\cdot(\pk-\qqk)+c_{20}\b^2\bI\qk+c_{21}\b^2\bI\tk+c_{22}\bI(\roohk-\rkoho)+c_{23}\b\bI\rohk+c_{24}\b\bI\sohk\\
		&&+c_{25}\b\bI\rkho+c_{26}\b\bI\skho+(c_{2701}\a^2+c_{2702}\b^2)\bI\rhk+(c_{2801}\roo+c_{2802}\b\ro+c_{2803}\\
		&&\cdot\b\so+c_{2804}\a^2r+c_{2805}\b^2r)\rIk+c_{29}r\rIo\yk+(c_{3001}\ro+c_{3002}\so+c_{3003}\b r)\rIo\bk\\
		&&+c_{31}\rIo\rko+c_{32}\b\rIo\rk+c_{33}\b\rIo\sk+c_{34}(\ro+\so)\sIo\bk+c_{35}\sIo\sko+c_{36}\b\sIo(\rk\\
		&&+\sk)+\big\{c_{3701}(\ro+\so)+c_{3702}\b r\big\}\rI\yk+(c_{3801}\roo+c_{3802}\b\ro+c_{3803}\b\so+c_{3804}\\
		&&\cdot\a^2 r)\rI\bk+c_{39}\b\rI\rko+c_{40}\b\rI\sko+(c_{4101}\a^2+c_{4102}\b^2)\rI\rk+(c_{4201}\a^2+c_{4202}\b^2)\\
		&&\cdot\rI\sk+\big\{c_{4301}(\ro+\so)+c_{4302}\b r\big\}\sI\yk+(c_{4401}\roo+c_{4402}\b\ro+c_{4403}\b\so+c_{4404}\a^2\\
		&&\cdot r)\sI\bk+c_{45}\b\sI\rko+c_{46}\b\sI\sko+(c_{4701}\a^2+c_{4702}\b^2)\sI\rk+(c_{4801}\a^2+c_{4802}\b^2)\sI\sk\\
		&&+c_{49}\b^2\tIk+c_{50}\b\tIo\bk+c_{51}(\qI+\tI)(\a^2\bk-\b\yk)+c_{52}(\b\sIkho-2\b\sIohk+\sIoho\bk)\\
		&&+(c_{5301}\a^2+c_{5302}\b^2)(\rIhk+\sIhk)+c_{54}(\rIho+\sIho)\yk+c_{55}\b(\rIho+\sIho)\bk,
	\end{eqnarray*}
	where
	\begin{eqnarray*}
		&c_{101}=-\f{4\kappa\rho'^2}{1-b^2\kappa},\,\, c_{102}=-4\rho'^2,\,\, c_{103}=\f{2\kappa}{1-b^2\kappa}(\kappa'+2\kappa\rho')\rho',\,\, c_{104}=2(\kappa'+2\kappa\rho')\rho',&\\
		&c_{105}=-\f{4\kappa'\rho'}{1-b^2\kappa},\,\,
		c_{106}=-\f{4}{1-b^2\kappa}(\kappa^2+\kappa')\rho',\,\,
		c_{107}=-4\kappa\rho',\,\,
	    c_{108}=-\f{2\kappa\rho'}{1-b^2\kappa},&\\
		&c_{109}=4(\rho'^2-\rho''),\,\,
		c_{1010}=-2\rho',\,\,
		c_{201}=\f{4\kappa}{1-b^2\kappa}\rho'^2,\,\,
		c_{202}=4\rho'^2,\,\,
		c_{301}=\f{2\kappa'\rho'}{1-b^2\kappa},&\\
		&c_{302}=\f{2}{1-b^2\kappa}(\kappa^2+\kappa')\rho',\,\,
		c_{303}=2\kappa\rho',\,\,
		c_{304}=-\f{2\kappa}{1-b^2\kappa}(\kappa'+2\kappa\rho')\rho',&\\
		&c_{305}=-2(\kappa'+2\kappa\rho')\rho',\,\,
		c_4=\f{2\kappa\rho'}{1-b^2\kappa},\,\,
		c_{501}=-4(\rho'^2-\rho''),\,\,
		c_{502}=\f{2\kappa'\rho'}{1-b^2\kappa},&\\
		&c_{601}=-4(\rho'^2-\rho''),\,\,
		c_{602}=\f{2}{1-b^2\kappa}(\kappa^2+\kappa')\rho',\,\,
		c_7=2\kappa\rho',\,\,
		c_8=4\rho',\,\,
		c_9=-2\rho',&\\
		&c_{1001}=\f{2}{(1-b^2\kappa)^2}\{(\kappa^2+\kappa')\rho'-2\kappa(1-b^2\kappa)(\rho'^2-\rho'')\},\,\,
		c_{1002}=-\f{2\kappa\rho'}{1-b^2\kappa},&\\
		&c_{1003}=-\f{2\kappa'\rho'}{1-b^2\kappa},\,\,
		c_{1004}=\f{2}{1-b^2\kappa}(\kappa^2+2\kappa')\rho',\,\,
		c_{1005}=\f{2}{1-b^2\kappa}(\kappa^2+\kappa')\rho',&\\
		&c_{1006}=\f{2\kappa\rho'}{1-b^2\kappa},\,\,
		c_{1101}=-\f{\kappa^2}{1-b^2\kappa},\,\,
		c_{1102}=\f{1}{(1-b^2\kappa)^2}\{\kappa\kappa'+b^2\kappa'^2+2(1-b^2\kappa)\kappa''\},&\\
		&c_{1103}=\f{1}{(1-b^2\kappa)^2}\{\kappa^3+6\kappa\kappa'-3b^2\kappa^2\kappa'+2b^2\kappa'^2
		+4(1-b^2\kappa)\kappa''\},&\\
		&c_{1104}=\f{1}{(1-b^2\kappa)^2}\{\kappa^3+5\kappa\kappa'-3b^2\kappa^2\kappa'+b^2\kappa'^2
		+2(1-b^2\kappa)\kappa''\},&\\
		&c_{1105}=\f{\kappa'}{1-b^2\kappa},\,\,
		c_{1106}=\f{\kappa^2+\kappa'}{1-b^2\kappa},&\\
		&c_{1107}=-\f{\kappa}{(1-b^2\kappa)^2}\{\kappa\kappa'+b^2\kappa'^2+2(1-b^2\kappa)\kappa''+2(\kappa^2+\kappa')\rho'-4\kappa(1-b^2\kappa)(\rho'^2-\rho'')\},&\\
		&c_{1108}=\f{\kappa}{1-b^2\kappa}(\kappa'+2\kappa\rho'),\,\,
		c_{1109}=\f{\kappa\kappa'}{1-b^2\kappa},\,\,
		c_{1110}=\f{\kappa}{1-b^2\kappa}(\kappa^2+\kappa'),&\\
		&c_{1111}=-\f{\kappa}{1-b^2\kappa}(\kappa'+2\kappa\rho'),\,\,
		c_{1112}=\f{2\kappa'\rho'}{1-b^2\kappa},\,\,
		c_{1113}=-\f{2}{1-b^2\kappa}(\kappa^2+2\kappa')\rho',&\\
		&c_{1114}=-\f{2}{1-b^2\kappa}(\kappa^2+\kappa')\rho',\,\,
		c_{12}=\f{1}{(1-b^2\kappa)^2}(\kappa^2+\kappa'),\,\,
		c_{1301}=\f{3\kappa'}{1-b^2\kappa},&\\
		&c_{1302}=\f{3}{1-b^2\kappa}(\kappa^2+\kappa'),\,\,
		c_{1303}=-\f{3\kappa\kappa'}{1-b^2\kappa},\,\,
		c_{1401}=-\f{1}{(1-b^2\kappa)^2}(\kappa^2+\kappa'),&\\
		&c_{1402}=-\f{1}{(1-b^2\kappa)^2}\{\kappa\kappa'+b^2\kappa'^2+2(1-b^2\kappa)\kappa''\},&\\
		&c_{1403}=-\f{1}{(1-b^2\kappa)^2}\{2\kappa^3+6\kappa\kappa'-3b^2\kappa^2\kappa'+b^2\kappa'^2+2(1-b^2\kappa)\kappa''\},&\\
		&c_{1404}=-\f{2}{(1-b^2\kappa)^2}\{(\kappa^2+\kappa')\rho'-2\kappa(1-b^2\kappa)(\rho'^2-\rho'')\},&\\
		&c_{1405}=\f{\kappa}{(1-b^2\kappa)^2}\{\kappa\kappa'+b^2\kappa'^2+2(1-b^2\kappa)\kappa''+2(\kappa^2+\kappa')\rho'-4\kappa(1-b^2\kappa)(\rho'^2-\rho'')\},&\\
		&c_{1501}=-\f{1}{(1-b^2\kappa)^2}(\kappa^2+\kappa'),\,\,
		c_{1502}= \f{1}{(1-b^2\kappa)^2}\{\kappa^3-b^2\kappa'^2-2(1-b^2\kappa)\kappa''\},&\\
		&c_{1503}= -\f{1}{(1-b^2\kappa)^2}\{\kappa^3+5\kappa\kappa'-3b^2\kappa^2\kappa'+b^2\kappa'^2+2(1-b^2\kappa)\kappa''\},&\\
		&c_{1504}=-\f{2}{(1-b^2\kappa)^2}\{(\kappa^2+\kappa')\rho'-2\kappa(1-b^2\kappa)(\rho'^2-\rho'')\},&\\
		&c_{1505}=-\f{2}{(1-b^2\kappa)^2}\{(\kappa^2+\kappa')\rho'-2\kappa(1-b^2\kappa)(\rho'^2-\rho'')\},&\\
		&c_{16}=\f{2\kappa^2}{1-b^2\kappa},\,\,
		c_{17}=-\f{\kappa^2}{1-b^2\kappa},\,\,
		c_{18}=\f{2\kappa\rho'}{1-b^2\kappa},\,\,
		c_{19}=-\f{\kappa}{1-b^2\kappa}(\kappa'+2\kappa\rho'),&\\
		&c_{20}=-\f{\kappa\kappa'}{1-b^2\kappa},\,\,
		c_{21}=-\f{\kappa}{1-b^2\kappa}(\kappa^2+\kappa'),\,\,
		c_{22}=-\f{\kappa}{1-b^2\kappa},\,\,
		c_{23}=-\f{2\kappa'}{1-b^2\kappa},&\\
		&c_{24}=-\f{2}{1-b^2\kappa}(\kappa^2+\kappa'),\,\,
		c_{25}=\f{\kappa'}{1-b^2\kappa},\,\,
		c_{26}=\f{1}{1-b^2\kappa}(\kappa^2+\kappa'),\,\,
		c_{2701}=-\f{2\kappa\rho'}{1-b^2\kappa},&\\
		&c_{2702}=\f{\kappa}{1-b^2\kappa}(\kappa'+2\kappa\rho'),\,\,
		c_{2801}=-\f{\kappa}{1-b^2\kappa},\,\,
		c_{2802}=-\f{2\kappa'}{1-b^2\kappa},&\\
		&c_{2803}=-\f{2}{1-b^2\kappa}(\kappa^2+\kappa'),\,\,
		c_{2804}=-\f{2\kappa\rho'}{1-b^2\kappa},\,\,
		c_{2805}=\f{\kappa}{1-b^2\kappa}(\kappa'+2\kappa\rho'),&\\
		&c_{29}=\f{2\kappa\rho'}{1-b^2\kappa},\,\,
		c_{3001}=\f{\kappa'}{1-b^2\kappa},\,\,
		c_{3002}=\f{1}{1-b^2\kappa}(\kappa^2+\kappa'),\,\,
		c_{3003}=-\f{\kappa}{1-b^2\kappa}(\kappa'+2\kappa\rho'),&\\
		&c_{31}=\f{\kappa}{1-b^2\kappa},\,\,
		c_{32}=\f{\kappa'}{1-b^2\kappa},\,\,
	    c_{33}=\f{1}{1-b^2\kappa}(\kappa^2+\kappa'),\,\,
		c_{34}=3\kappa',&\\
		&c_{35}=3\kappa,\,\,
		c_{36}=-3\kappa',\,\,
		c_{3701}=-4(\rho'^2-\rho''),\,\,
		c_{3702}=\f{2\kappa'\rho'}{1-b^2\kappa},\,\,
		c_{3801}=-\f{\kappa'}{1-b^2\kappa},&\\
		&c_{3802}=-\f{1}{1-b^2\kappa}\{b^2\kappa'^2+2(1-b^2\kappa)\kappa''-4\kappa(1-b^2\kappa)(\rho'^2-\rho'')\},&\\
		&c_{3803}=-\f{1}{1-b^2\kappa}\{\kappa\kappa'+b^2\kappa'^2+2(1-b^2\kappa)\kappa''
		-4\kappa(1-b^2\kappa)(\rho'^2-\rho'')\},&\\
		&c_{3804}=-\f{2\kappa'\rho'}{1-b^2\kappa},\,\,
		c_{39}=\f{\kappa'}{1-b^2\kappa},\,\,
		c_{40}=-3\kappa',\,\,
		c_{4101}=4(\rho'^2-\rho''),&\\
		&c_{4102}=\f{1}{1-b^2\kappa}\{b^2\kappa'^2+2(1-b^2\kappa)\kappa''-4\kappa(1-b^2\kappa)(\rho'^2-\rho'')\},&\\
		&c_{4201}=4(\rho'^2-\rho''),&\\
		&c_{4202}=\f{1}{1-b^2\kappa}\{\kappa\kappa'+b^2\kappa'^2+2(1-b^2\kappa)\kappa''
		-4\kappa(1-b^2\kappa)(\rho'^2-\rho'')\},&\\
		&c_{4301}=-4(\rho'^2-\rho''),\,\,
		c_{4302}=\f{2}{1-b^2\kappa}(\kappa^2+\kappa')\rho',\,\,
		c_{4401}=-\f{1}{1-b^2\kappa}(\kappa^2+\kappa'),&\\
		&c_{4402}=-\f{1}{1-b^2\kappa}\{\kappa\kappa'+b^2\kappa'^2+2(1-b^2\kappa)\kappa''
		-4\kappa(1-b^2\kappa)(\rho'^2-\rho'')\},&\\
		&c_{4403}=-\f{1}{1-b^2\kappa}\{\kappa^3+2\kappa\kappa'+b^2\kappa'^2+2(1-b^2\kappa)\kappa''
		-4\kappa(1-b^2\kappa)
		(\rho'^2-\rho'')\},&\\
		&c_{4404}=-\f{2}{1-b^2\kappa}(\kappa^2+\kappa')\rho',\,\,
		c_{45}=\f{1}{1-b^2\kappa}(\kappa^2+\kappa'),\,\,
		c_{46}=-3\kappa',\,\,
		c_{4701}=4(\rho'^2-\rho''),&\\
		&c_{4702}=\f{1}{1-b^2\kappa}\{\kappa\kappa'+b^2\kappa'^2+2(1-b^2\kappa)\kappa''
		-4\kappa(1-b^2\kappa)(\rho'^2-\rho'')\},&\\
		&c_{4801}=4(\rho'^2-\rho''),&\\
		&c_{4802}=\f{1}{1-b^2\kappa}\{\kappa^3+2\kappa\kappa'+b^2\kappa'^2+2(1-b^2\kappa)\kappa''-4\kappa(1-b^2\kappa)
		(\rho'^2-\rho'')\},&\\
		&c_{49}=-\kappa^2,\,\,
		c_{50}=\kappa^2,\,\,
		c_{51}=-2\kappa\rho',\,\,
		c_{52}=\kappa,&\\
		&c_{5301}=-2\rho',\,\,
		c_{5302}=\kappa'+2\kappa\rho',\,\,
	    c_{54}=2\rho',\,\,
		c_{55}=-(\kappa'+2\kappa\rho').&
	\end{eqnarray*}
\end{proposition}

\begin{proposition}\label{Rbbunderdeformaion}
	After $\b$-deformation,
	\begin{eqnarray*}
		\bRbb&=&c_1\Rbb+c_{2}\roo r+c_{3}\ro^2+c_{4}\ro\so+c_{5}\so^2+c_{6}\poo+c_{7}\qoo
		+c_{8}\too+c_{9}\roohb+c_{10}\roho\\
		&&+c_{11}\soho+c_{12}\b\ro r+c_{13}\b\so r+c_{14}\b\qo+c_{15}\b\to+c_{16}\b\rohb+c_{17}\b\sohb+c_{18}\a^2r^2\\
		&&+c_{19}\b^2r^2+c_{20}\a^2p+c_{21}\b^2p+c_{22}\a^2q+c_{23}\a^2q+c_{24}\a^2t+c_{25}\b^2t+c_{26}\a^2\rhb\\
		&&+c_{27}\b^2\rhb,
	\end{eqnarray*}
	where
	\begin{eqnarray*}
		&c_1=\f{e^{-2\rho}\nu^2}{1-b^2\kappa},\,\,
		c_{2}=-\f{e^{-2\rho}\nu^2}{(1-b^2\kappa)^3}\{\kappa+2b^2\kappa'-b^4\kappa\kappa'
		+2b^2(1-b^2\kappa)\kappa\rho'\},&\\
		&c_{3}=\f{e^{-2\rho}\nu^2}{(1-b^2\kappa)^3}\{\kappa+2b^2\kappa'+b^6\kappa'^2+2b^4(1-b^2\kappa)\kappa''
		+4b^2(1-b^2\kappa)^2
		(\rho'^2-\rho'')\},&\\
		&c_{4}=\f{2b^2e^{-2\rho}\nu^2}{(1-b^2\kappa)^3}\{\kappa^2+2b^2\kappa\kappa'+b^4\kappa'^2+2(1-b^2\kappa)(2\kappa'
		+b^2\kappa'')+4(1-b^2\kappa)^2(\rho'^2-\rho'')\},&\\
		&c_{5}=\f{e^{-2\rho}\nu^2}{(1-b^2\kappa)^3}\{\kappa(3-3b^2\kappa+b^4\kappa^2)+2b^2\kappa'+b^6\kappa'^2
		+2b^2(1-b^2\kappa)
		(2\kappa'+b^2\kappa'')&\\
		&+4b^2(1-b^2\kappa)^2(\rho'^2-\rho'')\},\,\,
		c_{6}=-\f{b^2\kappa e^{-2\rho}\nu^2}{(1-b^2\kappa)^2},\,\,
		c_{7}=-\f{2b^2\kappa e^{-2\rho}\nu^2}{(1-b^2\kappa)^2},\,\,
		c_{8}=\f{b^2\kappa e^{-2\rho}\nu^2}{1-b^2\kappa},&\\
		&c_{9}=-\f{b^2\kappa e^{-2\rho}\nu^2}{(1-b^2\kappa)^2},\,\,
		c_{10}=\f{b^2e^{-2\rho}\nu^2}{(1-b^2\kappa)^2}\{\kappa+b^2\kappa'-2(1-b^2\kappa)\rho'\},&\\
		&c_{11}=\f{b^2e^{-2\rho}\nu^2}{(1-b^2\kappa)^2}\{\kappa+b^2\kappa'-2(1-b^2\kappa)\rho'\},\,\,
		c_{12}=-\f{2e^{-2\rho}\nu^2}{(1-b^2\kappa)^3}\{b^2(\kappa+b^2\kappa')\kappa'&\\
		&+2b^2(1-b^2\kappa)\kappa''-2\kappa(1-b^2\kappa)\rho'+4(1-b^2\kappa)^2(\rho'^2-\rho'')\},&\\
		&c_{13}=-\f{2e^{-2\rho}\nu^2}{(1-b^2\kappa)^3}\{(\kappa+b^2\kappa')^2+2(1-b^2\kappa)(2\kappa'+b^2\kappa'')
		+4(1-b^2\kappa)^2(\rho'^2-\rho'')\},&\\
		&c_{14}=\f{2\kappa e^{-2\rho}\nu^2}{(1-b^2\kappa)^2},\,\,
		c_{15}=-\f{2\kappa e^{-2\rho}\nu^2}{1-b^2\kappa},\,\,
		c_{16}=-\f{2e^{-2\rho}\nu^2}{(1-b^2\kappa)^2}\{b^2\kappa'-2(1-b^2\kappa)\rho'\},&\\
		&c_{17}=-\f{2e^{-2\rho}\nu^2}{(1-b^2\kappa)^2}\{\kappa+b^2\kappa'-2(1-b^2\kappa)\rho'\},&\\
		&c_{18}=-\f{2e^{-2\rho}\nu^2}{(1-b^2\kappa)^3}\{(\kappa+2b^2\kappa'-b^4\kappa\kappa')\rho'-2(1-b^2\kappa)^2\rho'^2+2(1-b^2\kappa)\rho''\},&\\
		&c_{19}=\f{e^{-2\rho}\nu^2}{(1-b^2\kappa)^3}\{(\kappa+b^2\kappa')\kappa'+2(1-b^2\kappa)\kappa''+2(\kappa^2
		+2\kappa'-b^2\kappa\kappa')
		\rho'+4\kappa(1-b^2\kappa)\rho''\},&\\
		&c_{20}=\f{2e^{-2\rho}\nu^2}{(1-b^2\kappa)^2}\rho'\{1+b^4\kappa'-2b^2(1-b^2\kappa)\rho'\},&\\
		&c_{21}=-\f{e^{-2\rho}\nu^2}{(1-b^2\kappa)^2}\{\kappa'+2(\kappa+b^2\kappa')\rho'-4(1-b^2\kappa)\rho'^2\},&\\
		&c_{22}=-\f{4b^2e^{-2\rho}\nu^2}{(1-b^2\kappa)^2}\rho'\{\kappa+b^2\kappa'-2(1-b^2\kappa)\rho'\},&\\
		&c_{23}=\f{4e^{-2\rho}\nu^2}{(1-b^2\kappa)^2}\rho'\{\kappa+b^2\kappa'-2(1-b^2\kappa)\rho'\},&\\
		&c_{24}=\f{2e^{-2\rho}\nu^2}{(1-b^2\kappa)^2}\rho'\{1-2b^2\kappa-b^4\kappa'+2b^2(1-b^2\kappa)\rho'\},&\\
		&c_{25}=-\f{e^{-2\rho}\nu^2}{(1-b^2\kappa)^2}\{\kappa^2+\kappa'-2(\kappa+b^2\kappa')\rho'+4(1-b^2\kappa)\rho'^2\},&\\
		&c_{26}=-\f{2e^{-2\rho}\nu^2}{(1-b^2\kappa)^2}\rho',\,\,
		c_{27}=\f{e^{-2\rho}\nu^2}{(1-b^2\kappa)^2}(\kappa'+2\kappa\rho').&
	\end{eqnarray*}
\end{proposition}

\begin{proposition}\label{Ricoounderdeformation}
	After $\b$-deformation,
	\begin{eqnarray*}
		\bRicoo&=&\Ricoo+C_1\roo\rIi+C_2\roo r+C_3\ro^2+C_4\ro\so+C_5\so^2+C_6\qoo+C_7\too+C_8\roohb\\
		&&+C_9\roho+C_{10}\soho+C_{11}\b\ro\rIi+C_{12}\b\so\rIi+C_{13}\b\ro r+C_{14}\b\so r+C_{15}\b\po\\
		&&+C_{16}\b\qqo+C_{17}\b\qo+C_{18}\b\to+C_{19}\b\sIohi+C_{20}\b\rohb+C_{21}\b\sohb+C_{22}\a^2\rIi r\\
		&&+C_{23}\b^2\rIi r+C_{24}\a^2 r^2+C_{25}\b^2 r^2+C_{26}\b^2\tIi+C_{27}\a^2 p+C_{28}\b^2 p+C_{29}\a^2 q\\
		&&+C_{30}\b^2 q+C_{31}\a^2 t+C_{32}\b^2 t+C_{33}\a^2\rIhi+C_{34}\b^2\rIhi+C_{35}\a^2\sIhi+C_{36}\b^2\sIhi\\
		&&+C_{37}\a^2\rhb+C_{38}\b^2\rhb,
	\end{eqnarray*}
	where
	\begin{eqnarray*}
		&C_1=-\f{\kappa}{1-b^2\kappa},\,\,
		C_2=-\f{1}{(1-b^2\kappa)^2}\{2(1-b^2\kappa)\kappa'+\kappa(\kappa+b^2\kappa')+2(n-2)\kappa(1-b^2\kappa)\rho'\},&\\
		&C_3=\f{1}{(1-b^2\kappa)^2}\{\kappa^2+2\kappa'+b^4\kappa'^2+2b^2(1-b^2\kappa)\kappa''
		+4(n-2)(1-b^2\kappa)^2(\rho'^2-\rho'')\},&\\
		&C_4= \f{2}{(1-b^2\kappa)^2}\{\kappa^2+4\kappa'-b^2\kappa\kappa'+b^4\kappa'^2+2b^2(1-b^2\kappa)\kappa''+4(n-2)(1-b^2\kappa)^2&\\
		&\cdot(\rho'^2-\rho'')\},\,\,
		C_5=-\f{1}{(1-b^2\kappa)^2}\{2b^2\kappa^3-3\kappa^2-6\kappa'+4b^2\kappa\kappa'-b^4\kappa'^2-2b^2(1-b^2\kappa)\kappa''&\\
		&-4(n-2)(1-b^2\kappa)^2(\rho'^2-\rho'')\},\,\,
		C_6=-\f{2\kappa}{1-b^2\kappa},\,\,
		C_7=-2\kappa,\,\,
		C_8=-\f{\kappa}{1-b^2\kappa},&\\
		&C_9=C_{10}=\f{1}{1-b^2\kappa}\{\kappa+b^2\kappa'-2(n-2)(1-b^2\kappa)\rho'\},&\\
		&C_{11}=-\f{2}{1-b^2\kappa}\kappa',\,\,
		C_{12}=-\f{2}{1-b^2\kappa}(\kappa^2+\kappa'),&\\
		&C_{13}=-\f{2}{(1-b^2\kappa)^2}\{(\kappa+b^2\kappa')\kappa'+2(1-b^2\kappa)\kappa''
		+2(n-2)(1-b^2\kappa)\kappa'\rho'\},&\\
		&C_{14}=-\f{2}{(1-b^2\kappa)^2}\{5\kappa(1-b^2\kappa)\kappa'+\kappa^2(\kappa+b^2\kappa')
		+b^2(\kappa^2+\kappa')\kappa'+2(1-b^2\kappa)\kappa''&\\
		&+2(n-2)(1-b^2\kappa)(\kappa^2+\kappa')\rho'\},\,\,
		C_{15}=\f{2}{1-b^2\kappa}\kappa',\,\,
		C_{16}=-\f{2}{1-b^2\kappa}(\kappa^2+\kappa'),&\\
		&C_{17}=\f{2}{1-b^2\kappa}\{\kappa^2-3(1-b^2\kappa)\kappa'-2(n-2)\kappa(1-b^2\kappa)\rho'\},&\\
		&C_{18}=-6\kappa'-4(n-2)\kappa\rho',\,\,
		C_{19}=-2\kappa,\,\,
		C_{20}=-\f{2}{1-b^2\kappa}\kappa',&\\
		&C_{21}=-\f{2}{1-b^2\kappa}(\kappa^2+\kappa'),\,\,
		C_{22}=-\f{2\kappa}{1-b^2\kappa}\rho',\,\,
		C_{23}=\f{\kappa}{1-b^2\kappa}(\kappa'+2\kappa\rho'),&\\
		&C_{24}=-\f{2}{(1-b^2\kappa)^2}\{2(1-b^2\kappa)\kappa'\rho'+\kappa(\kappa+b^2\kappa')\rho'
		+2(n-2)\kappa(1-b^2\kappa)\rho'^2&\\
		&+2\kappa(1-b^2\kappa)\rho''\},\,\,
		C_{25}=\f{\kappa}{(1-b^2\kappa)^2}\{(\kappa+b^2\kappa')\kappa'+2(1-b^2\kappa)\kappa''
		&\\
		&+2(n-1)(1-b^2\kappa)\kappa'\rho'+4(n-2)\kappa(1-b^2\kappa)\rho'^2+2(\kappa^2+\kappa')\rho'+4\kappa(1-b^2\kappa)\rho''\},&\\
		&C_{26}=-\kappa^2,\,\,
		C_{27}=\f{2}{1-b^2\kappa}\{(\kappa+b^2\kappa')\rho'-2(n-2)(1-b^2\kappa)\rho'^2-2(1-b^2\kappa)\rho''\},&\\
		&C_{28}=-\f{1}{1-b^2\kappa}\{(\kappa-b^2\kappa')\kappa'-2(1-b^2\kappa)\kappa''+2\kappa(\kappa+b^2\kappa')\rho'-2(n-2)(1-b^2\kappa)\kappa'\rho'&\\
		&-4(n-2)\kappa(1-b^2\kappa)\rho'^2-4\kappa(1-b^2\kappa)\rho''\},\,\,
		C_{29}=-\f{4}{1-b^2\kappa}\{(\kappa+b^2\kappa')\rho'&\\
		&-2(n-2)(1-b^2\kappa)\rho'^2-2(1-b^2\kappa)\rho''\},\,\, C_{30}=-\f{2}{1-b^2\kappa}\{(\kappa+b^2\kappa')\kappa'+2(1-b^2\kappa)\kappa''&\\
		&-2\kappa(\kappa+b^2\kappa')\rho'+2(n-2)(1-b^2\kappa)\kappa'\rho'+4(n-2)\kappa(1-b^2\kappa)\rho'^2+4\kappa(1-b^2\kappa)\rho''\},&\\
		&C_{31}=-\f{2}{1-b^2\kappa}\{(\kappa+b^2\kappa')\rho'-2(n-2)(1-b^2\kappa)\rho'^2-2(1-b^2\kappa)\rho''\},&\\
		&C_{32}=-\f{1}{1-b^2\kappa}\{2\kappa^3+(\kappa+b^2\kappa')\kappa'+2(1-b^2\kappa)\kappa''-2\kappa(\kappa+b^2\kappa')\rho'&\\
		&+2(n-2)(1-b^2\kappa)\kappa'\rho'+4(n-2)\kappa(1-b^2\kappa)\rho'^2+4\kappa(1-b^2\kappa)\rho''\},&\\
		&C_{33}=-2\rho',\,\,
		C_{34}=\kappa'+2\kappa\rho',\,\,
		C_{35}=-2\rho',\,\,
		C_{36}=\kappa'+2\kappa\rho',&\\
		&C_{37}=-\f{2\kappa}{1-b^2\kappa}\rho',\,\,
		C_{38}=\f{\kappa}{1-b^2\kappa}(\kappa'+2\kappa\rho').&
	\end{eqnarray*}
\end{proposition}

\noindent Changtao Yu\\
School of Mathematical Sciences, South China Normal
University\\
Guangzhou, 510631, P. R. China\\
aizhenli@gmail.com
\end{document}